\documentclass[12pt]{amsart}
\usepackage{amsmath,amssymb,amsbsy,amsfonts,latexsym,amsopn,amstext,cite,
                                               amsxtra,euscript,amscd,bm}
\usepackage{url}

\usepackage{mathrsfs}
 
\usepackage[papersize={9in,11in},textwidth=15.8cm,textheight=21.5cm,centering]{geometry}

\allowdisplaybreaks

\usepackage{color}
\usepackage[colorlinks,linkcolor=blue,anchorcolor=blue,citecolor=blue,backref=page]{hyperref}
\usepackage{graphics,epsfig}
\usepackage{graphicx}
\usepackage{float}
\usepackage{epstopdf}
\hypersetup{breaklinks=true}

\def\tk{\widetilde k} 
\def\tm{\widetilde m} 
\def\tn{\widetilde n} 
\def\tu{\widetilde u}

\usepackage[np]{numprint}
\npdecimalsign{\ensuremath{.}}

\usepackage{bibentry}

\usepackage[english]{babel}
\usepackage{mathtools}
\usepackage{todonotes}
\usepackage{url}
\usepackage[colorlinks,linkcolor=blue,anchorcolor=blue,citecolor=blue,backref=page]{hyperref}

\usepackage[norefs,nocites]{refcheck}

\begin{document}

\newcommand{\bs}{\boldsymbol}
\def \a{\alpha} \def \b{\beta} \def \d{\delta} \def \e{\varepsilon} \def \g{\gamma} \def \k{\kappa} \def \l{\lambda} \def \s{\sigma} \def \t{\theta} \def \z{\zeta}

\newcommand{\mb}{\mathbb}

\newtheorem{theorem}{Theorem}
\newtheorem{lemma}[theorem]{Lemma}
\newtheorem{claim}[theorem]{Claim}
\newtheorem{corollary}[theorem]{Corollary}
\newtheorem{conj}[theorem]{Conjecture}
\newtheorem{prop}[theorem]{Proposition}
\newtheorem{definition}[theorem]{Definition}
\newtheorem{question}[theorem]{Question}
\newtheorem{example}[theorem]{Example}
\newcommand{\hh}{{{\mathrm h}}}
\newtheorem{remark}[theorem]{Remark}

\numberwithin{equation}{section}
\numberwithin{theorem}{section}
\numberwithin{table}{section}
\numberwithin{figure}{section}

\def\sssum{\mathop{\sum\!\sum\!\sum}}
\def\ssum{\mathop{\sum\ldots \sum}}

\newcommand{\diam}{\operatorname{diam}}

\def\squareforqed{\hbox{\rlap{$\sqcap$}$\sqcup$}}
\def\qed{\ifmmode\squareforqed\else{\unskip\nobreak\hfil
\penalty50\hskip1em \nobreak\hfil\squareforqed
\parfillskip=0pt\finalhyphendemerits=0\endgraf}\fi}

\newfont{\teneufm}{eufm10}
\newfont{\seveneufm}{eufm7}
\newfont{\fiveeufm}{eufm5}
%
%
\newfam\eufmfam
     \textfont\eufmfam=\teneufm
\scriptfont\eufmfam=\seveneufm
     \scriptscriptfont\eufmfam=\fiveeufm
%
%
\def\frak#1{{\fam\eufmfam\relax#1}}

\newcommand{\bflambda}{{\boldsymbol{\lambda}}}
\newcommand{\bfmu}{{\boldsymbol{\mu}}}
\newcommand{\bfxi}{{\boldsymbol{\eta}}}
\newcommand{\bfrho}{{\boldsymbol{\rho}}}

\def\eps{\varepsilon}

\def\sssum{\mathop{\sum\!\sum\!\sum}}
\def\ssum{\mathop{\sum\ldots \sum}}
\def\dsum{\mathop{\quad \sum \qquad \sum}}

\def\T {\mathsf {T}}
\def\Tor{\mathsf{T}_d}
\def\Tore{\widetilde{\mathrm{T}}_{d} }

\def\sM {\mathsf {M}}
\def\sL {\mathsf {L}}
\def\sK {\mathsf {K}}
\def\sP {\mathsf {P}}

\def\ss{\mathsf {s}}

\def \balpha{\bm{\alpha}}
\def \bbeta{\bm{\beta}}
\def \bgamma{\bm{\gamma}}
\def \bdelta{\bm{\delta}}
\def \bzeta{\bm{\zeta}}
\def \blambda{\bm{\lambda}}
\def \bchi{\bm{\chi}}
\def \bphi{\bm{\varphi}}
\def \bpsi{\bm{\psi}}
\def \bxi{\bm{\xi}}
\def \bmu{\bm{\mu}}
\def \bnu{\bm{\nu}}
\def \bomega{\bm{\omega}}

\def \ba{\mathbf{a}}
\def \bb{\mathbf{b}}

\def \bell{\bm{\ell}}

\def\eqref#1{(\ref{#1})}

\def\vec#1{\mathbf{#1}}

\def\rf#1{\left\lceil#1\right\rceil}

\newcommand{\abs}[1]{\left| #1 \right|}

\def\Zq{\mathbb{Z}_q}
\def\Zqx{\mathbb{Z}_q^*}
\def\Zd{\mathbb{Z}_d}
\def\Zdx{\mathbb{Z}_d^*}
\def\Zf{\mathbb{Z}_f}
\def\Zfx{\mathbb{Z}_f^*}
\def\Zp{\mathbb{Z}_p}
\def\Zpx{\mathbb{Z}_p^*}
\def\cM{\mathcal M}
\def\cE{\mathcal E}
\def\cH{\mathcal H}

\def\le{\leqslant}
\def\leq{\leqslant}
\def\ge{\geqslant}
\def\geq{\geqslant}

\def\sfB{\mathsf {B}}
\def\sfC{\mathsf {C}}
\def\sfS{\mathsf {S}}
\def\sfI{\mathsf {I}}
\def\sfT{\mathsf {T}}
\def\L{\mathsf {L}}
\def\FF{\mathsf {F}}

\def\sE {\mathscr{E}}
\def\sS {\mathscr{S}}

\def\cA{{\mathcal A}}
\def\cB{{\mathcal B}}
\def\cC{{\mathcal C}}
\def\cD{{\mathcal D}}
\def\cE{{\mathcal E}}
\def\cF{{\mathcal F}}
\def\cG{{\mathcal G}}
\def\cH{{\mathcal H}}
\def\cI{{\mathcal I}}
\def\cJ{{\mathcal J}}
\def\cK{{\mathcal K}}
\def\cL{{\mathcal L}}
\def\cM{{\mathcal M}}
\def\cN{{\mathcal N}}
\def\cO{{\mathcal O}}
\def\cP{{\mathcal P}}
\def\cQ{{\mathcal Q}}
\def\cR{{\mathcal R}}
\def\cS{{\mathcal S}}
\def\cT{{\mathcal T}}
\def\cU{{\mathcal U}}
\def\cV{{\mathcal V}}
\def\cW{{\mathcal W}}
\def\cX{{\mathcal X}}
\def\cY{{\mathcal Y}}
\def\cZ{{\mathcal Z}}
\newcommand{\rmod}[1]{\: \mbox{mod} \: #1}

\def\cg{{\mathcal g}}

\def\vy{\mathbf y}
\def\vr{\mathbf r}
\def\vx{\mathbf x}
\def\va{\mathbf a}
\def\vb{\mathbf b}
\def\vc{\mathbf c}
\def\ve{\mathbf e}
\def\vf{\mathbf f}
\def\vg{\mathbf g}
\def\vh{\mathbf h}
\def\vk{\mathbf k}
\def\vm{\mathbf m}
\def\vz{\mathbf z}
\def\vu{\mathbf u}
\def\vv{\mathbf v}

\newcommand{\fa}{\mathfrak{a}}
\newcommand{\fb}{\mathfrak{b}}
\newcommand{\fc}{\mathfrak{c}}
\newcommand{\fl}{\mathfrak{l}}
\newcommand{\fm}{\mathfrak{m}}
\newcommand{\fn}{\mathfrak{n}}
\newcommand{\fq}{\mathfrak{q}}
\newcommand{\fA}{\mathfrak{A}}
\newcommand{\fB}{\mathfrak{B}}
\newcommand{\fH}{\mathfrak{H}}
\newcommand{\fJ}{\mathfrak{J}}
\newcommand{\fK}{\mathfrak{K}}
\newcommand{\fL}{\mathfrak{L}}
\newcommand{\fR}{\mathfrak{R}}
\newcommand{\fQ}{\mathfrak{Q}}
\newcommand{\fS}{\mathfrak{S}}
\newcommand{\fT}{\mathfrak{T}}
\newcommand{\fZ}{\mathfrak{Z}}

\def\e{{\mathbf{\,e}}}
\def\ep{{\mathbf{\,e}}_p}
\def\eq{{\mathbf{\,e}}_q}
\def\er{{\mathbf{\,e}}_r}
\def\es{{\mathbf{\,e}}_s}

 \def\SS{{\mathbf{S}}}

 \def\0{{\mathbf{0}}}
 
 \newcommand{\GL}{\operatorname{GL}}
\newcommand{\SL}{\operatorname{SL}}
\newcommand{\lcm}{\operatorname{lcm}}
\newcommand{\ord}{\operatorname{ord}}
\newcommand{\Tr}{\operatorname{Tr}}
\newcommand{\Span}{\operatorname{Span}}

\def\({\left(}
\def\){\right)}
\def\l|{\left|}
\def\r|{\right|}
\def\floor#1{\left\lfloor#1\right\rfloor}
\def\sumstar#1{\mathop{\sum\vphantom|^{\!\!*}\,}_{#1}}

\def\mand{\qquad \mbox{and} \qquad}



\hyphenation{re-pub-lished}

\mathsurround=1pt

\def\bfdefault{b}

\def \F{{\mathbb F}}
\def \K{{\mathbb K}}
\def \N{{\mathbb N}}
\def \Z{{\mathbb Z}}
\def \P{{\mathbb P}}
\def \Q{{\mathbb Q}}
\def \R{{\mathbb R}}
\def \C{{\mathbb C}}
\def\Fp{\F_p}
\def \fp{\Fp^*}

\newcommand{\ud}{\mathrm{d}}
\newcommand{\ue}{\mathrm{e}}

 \def \xbar{\overline x}

\title{Estimates for trilinear and quadrilinear character sums}

\author[\'E. Fouvry] {\'Etienne Fouvry}
\address{D{\'e}partement de Math{\'e}matiques, Universit{\'e} Paris-Saclay,
91405 Orsay Cedex, 
France}
\email{etienne.fouvry@universite-paris-saclay.fr}

\author[I. E. Shparlinski] {Igor E. Shparlinski}
\address{School of Mathematics and Statistics, University of New South Wales, Sydney NSW 2052, Australia}
\email{igor.shparlinski@unsw.edu.au}

\author[P. Xi]{Ping Xi}

\address{School of Mathematics and Statistics, Xi'an Jiaotong University, Xi'an 710049, P. R. China}
\email{ping.xi@xjtu.edu.cn}

\date{\today}

\begin{abstract} 
We obtain new bounds on some trilinear and quadrilinear character sums, which 
are non-trivial starting from very short ranges of the variables. 
 An  application to an apparently new problem on oscillations
of characters on differences between Farey fractions is given. 
Other applications include
a modular analogue of a multiplicative hybrid problem of Iwaniec and S{\'a}rk{\"o}zy~(1987)
and the solvability of some prime type equations with constraints.

\end{abstract}

\subjclass[2020]{11L40}

\keywords{character sums, trilinear forms, quadrilinear forms}

\maketitle

\tableofcontents

\section{Introduction and backgrounds}
\subsection{Set-up}
Motivated by various applications to analytic number theory, estimates for character sums received a lot of attention in the past decades. In many situations, the variables are supported over some additively structured sets, such as  sets of consecutive integers. But the difficulty can vary significantly since the weights might be arbitrary.
This paper aims to study two kinds of multilinear character sums with 
arbitrary weights. 

Throughout this paper, denote by $p$ a large  prime, and by $\chi$ a non-trivial multiplicative character modulo $p$. Take two integers $a,b$ with $p\nmid ab$.
For $K,L,M,N\geqslant1$, we consider the trilinear character sum
 \begin{equation}\label{eq:trilinearcharactersum}
\fT(\balpha,\bbeta)
= \sum_{k\leqslant K} \sum_{m\leqslant M}\sum_{n\leqslant N}\alpha_m\beta_{k,n}\chi(ak+bmn)
\end{equation}
and the quadrilinear sum
 \begin{equation}\label{eq:quadrilinearcharactersum}
\fQ(\boldsymbol\alpha,\boldsymbol\beta)
=\sum_{k\leqslant K}\sum_{\ell\leqslant L} \sum_{m\leqslant M}\sum_{n\leqslant N} \alpha_{\ell, m}\beta_{k,n}\chi(ak \ell + bmn),
\end{equation}
where $\balpha=(\alpha_m)$ or $(\alpha_{\ell,m})$ and $\bbeta=(\beta_{k,n})$ are some complex weights.  Note that one obtains~\eqref{eq:trilinearcharactersum} if taking $L=1$ and $\alpha_{1,m}=\alpha_m$ 
in~\eqref{eq:quadrilinearcharactersum}. No confusion on the definitions of $\balpha$ arises since one is for the trilinear sum, and the other is for the quadrilinear one.

In various practical applications, one aims to obtain upper bounds, as strong as possible, for $\fT(\balpha,\bbeta)$ and $\fQ(\balpha,\bbeta)$ in terms of the $\ell_\rho$-norms $\|\balpha\|_\rho$ and $\|\bbeta\|_\rho$ (see~\eqref{eq:norm} below for the definition of norms).
We refer to the following inequalities
$$
|\fT(\balpha,\bbeta)|\leqslant \min\{\|\balpha\|_\infty\|\bbeta\|_\infty KMN,~  \|\balpha\|_2\|\bbeta\|_2 (KMN)^{\frac{1}{2}}\}
$$
and
$$
|\fQ(\balpha,\bbeta)|\leqslant \min\{\|\balpha\|_\infty\|\bbeta\|_\infty KLMN,~  \|\balpha\|_2\|\bbeta\|_2 (KLMN)^{\frac{1}{2}}\}
$$
as trivial bounds, which follow directly from the triangle inequality and the 
Cauchy--Schwarz inequality.
Our aim is to beat the above trivial bounds for 
$\balpha, \bbeta$ as general as possible and $K,L,M,N$ as small as possible compared to $p$.

We would like to emphasize that weights in $\fQ(\balpha,\bbeta)$ depend on variables from different products, which makes treatments of such sums much more difficult, as the standard smoothing technique  does not  immediately apply.

Our bounds for $\fT(\balpha,\bbeta)$ and $\fQ(\balpha,\bbeta)$ with power-savings, as well as applications, will be given in Section~\ref{sec:main}, and we would like to give a brief 
outline of related results right now.

\subsection{Related bilinear sums}
\label{sec:surv B}

The above studies on character sums over sumsets can be dated back to
Vinogradov~\cite[Chapter~V, Exercise~8.c]{Vi54} on the following bilinear form
\begin{equation}\label{eq:bilinearform-Vinogradov}
\fB(\balpha,\bbeta) = \sum_{m\in \cM} \sum_{n \in \cN} \alpha_m\beta_n\chi(m + n),
\end{equation}
where $\chi$ is a non-trivial Dirichlet character modulo $p$ and $\cM,\cN\subseteq[1,p]$ are arbitrary subsets with $M=\#\cM$ and $N=\#\cN$. A direct application of Fourier techniques yields
\begin{equation}\label{eq:Vinogradov}
|\fB(\balpha,\bbeta)| \le  \|\balpha\|_2\|\bbeta\|_2~ p^{\frac{1}{2}}.
\end{equation}
Although this bound is widely known, its full derivation is not easy to 
find in the literature, however it can be found as a short proof 
of~\cite[Equation~(1.4)]{SS17}.  
Despite a very short and elementary proof, the bound~\eqref{eq:Vinogradov}  has 
never been  improved in full generality. 
However,  Karatsuba~\cite{Ka92} (see also~\cite[Chapter~VIII, Problem~9]{Ka93}) 
has proved, that if
\begin{equation}\label{eq:PV-threshold}
M>p^{\frac{1}{2}+\eta} \mand  N>p^\eta
\end{equation}
for some $\eta>0$, then the inequality
\begin{equation}\label{eq:Karatsuba-implicit}
|\fB(\balpha,\bbeta)| \le  \|\balpha\|_\infty\|\bbeta\|_\infty MNp^{-\kappa}
\end{equation}
holds with some $\kappa>0$, depending only on $\eta$. The range~\eqref{eq:PV-threshold} reveals the P\'olya--Vinogradov threshold even when summing over arbitrary subsets. A similar phenomenon can also be found in~\cite{Xi23}.
We note that the proof of~\eqref{eq:Karatsuba-implicit}, in the range~\eqref{eq:PV-threshold},
does not seem to be 
present in literature, and a concise proof with an explicit upper bound, is provided in Appendix~\ref{appendix:Karatsuba}
to convince cautious readers.

It is worthwhile to point out that $\fB(\balpha,\bbeta)$ has received considerable attention in recent years due to its connection with the {\it 
Paley graph conjecture}. It is conjectured, for instance with just constant weights $\alpha_m = \beta_n =1$, that
$$
\fB(\balpha,\bbeta)=o(MN)
$$
for all subsets $\cM,\cN$ with $M,N>p^\eta$ and $\eta>0$. This is far from proven, and the classical inequalities of Vinogradov~\eqref{eq:Vinogradov} and Karatsuba~\eqref{eq:Karatsuba-implicit} still stand. See~\cite{DBSW21, HP20} for recent progress towards this conjecture.

For structural sets, such as intervals, or sets with small doubling, or sets supported on short intervals, a large variety of bounds for $\fB(\balpha,\bbeta)$ are known. 
For example, the above bounds of Vinogradov and Karatsuba have been improved by 
Friedlander and Iwaniec~\cite[Theorem~3]{FI93} when $\cM$ and $\cN$ 
are contained in intervals of lengths at most $\sqrt{p}$ and 
are of cardinalities $M,N \ge p^{\frac{11}{24}+\eta}$. Of course, the main point 
here is that $\frac{11}{24}<\frac{1}{2}$, which neither~\eqref{eq:Vinogradov} nor~\eqref{eq:Karatsuba-implicit} with~\eqref{eq:PV-threshold}
can achieve. The exponent $\frac{11}{24}$ has been reduced to $\frac{9}{20}$ by Bourgain, Garaev, Konyagin and Shparlinski~\cite[Theorem~25]{BGKS13}.
Several more results of this type can be found in~\cite{BGKS12, BKS15, Ch08, Ha17, Ka08, SS22, SS18, SV17, Vo18}.

\subsection{Related trilinear and quadrilinear sums}
\label{sec:surv T and Q} 
Before concluding this section, we also mention some recent works on various variants of $\fB(\balpha,\bbeta)$. To proceed, we assume $\cH,\cK,\cM,\cN\subseteq[1,p]$ are arbitrary subsets.
Hanson~\cite{Ha17} has studied
$$
\sum_{k\in\cK} \sum_{m\in \cM} \sum_{n \in \cN} \alpha_m\beta_{n}\chi(k+m + n),
$$
while Roche-Newton, Shparlinski and Winterhof\cite{RSW19} have considered
$$
\sum_{k \in \cK}  \sum_{m\in \cM} \sum_{n \in \cN} 
\chi(km+mn+nk)\ue\(\frac{a(km+mn+nk)}{p}\)
$$
with $\gcd(a,p)=1$. Shkredov and Shparlinski~\cite{SS17} have treated quadrilinear forms
\begin{align*}
&\sum_{h \in\cH} \sum_{k \in \cK}  \sum_{m\in \cM} \sum_{n \in \cN} 
 \alpha_{h} \beta_{k,m,n}\chi(h+ k + mn) ,\\
&\sum_{h \in\cH} \sum_{k \in \cK}  \sum_{m\in \cM} \sum_{n \in \cN} 
\alpha_{h} \beta_{k,m,n}\chi(h+ k(m+n)).
\end{align*}
See also\cite{KMPS20} for some recent generalizations and refinements.

Our main object $\fT(\balpha,\bbeta)$ is intimately related to the following trilinear character sum
$$
\sum_{k\leqslant K} \sum_{m \in \cM}\sum_{n \in \cN} 
\alpha_k\beta_m \gamma_n \chi(k+mn),
$$
where $\chi$ is a non-trivial Dirichlet character modulo $p$ and $\cM,\cN\subseteq[1,p]$ are arbitrary subsets with $M=\#\cM$ and $N=\#\cN$. Banks and Shparlinski~\cite[Theorem~2.2]{BS20} give 
the upper bound
$$
KMN(p^{-1}+(KM)^{-1}+K^{-2})^{\frac{1}{2r}}(p^{\frac{1}{4r}}+N^{-\frac{1}{2}}p^{\frac{1}{2r}})p^{o(1)}
$$
with an arbitrary fixed integer $r \ge 1$, provided that the three coefficients are all bounded.
Note that the above bound is non-trivial as long as
$$
K>p^{\frac{1}{4}+\eta},\qquad KM>p^{\frac{1}{2}+\eta},\qquad N>p^{\eta}
$$
with some constant $\eta>0$.

Our work on $\fT(\balpha,\bbeta)$ is largely inspired by the recent work of Fouvry and Shparlinski~\cite{FS23} on quadrilinear character sums such as 
\begin{align*}
& \fQ_1(\balpha,\bbeta)=\mathop{\sum\sum\sum\sum}_{1 \le r,s,u,v \le x}\alpha_r\beta_u\chi(rs-uv) ,\\
&\fQ_2(\bgamma)=\mathop{\sum\sum\sum\sum}_{1 \le r,s,u,v\le x}\gamma_{r,s,u}\chi(rs-uv), 
\end{align*}
with bounded complex weights $\balpha= (\alpha_r)$, $\bbeta=(\beta_u)$ and $\bgamma = (\gamma_{r,s,u})$.
It is shown in~\cite{FS23}, that  for any fixed $\eta > 0$ and 
$x\ge p^{\frac{1}{8}+\eta}$, we have 
$$
\fQ_1(\balpha,\bbeta)\ll x^{4 -\kappa} \mand  \fQ_2(\bgamma)\ll x^4   \frac{\log \log p}{\log p}
$$
where $\kappa>0$ and the implied constants depend only on $\eta$ (we refer to Section~\ref{sec:not}
for the exact definition of the symbol `$\ll$' and other standard notations).

\subsection{Notation and conventions} 
\label{sec:not}
We adopt the Landau symbol $A=O(B)$, and the Vinogradov symbol $A\ll B$, to mean $|A|\leqslant cB$ for some constant $c>0$. 
We also write $a \sim A$ to indicate that $A< a \le 2A$ and $A \asymp B$ to 
indicate $A \ll B \ll A$.



For each complex weight $\boldsymbol\alpha=(\alpha_m)_{m \in \cM}$ 
and $\rho\geqslant 1$, we denote 
\begin{equation}\label{eq:norm}
\|\balpha\|_\rho =\( \sum_{m \in \cM} |\alpha_m|^\rho\)^{1/\rho}
\mand 
\|\balpha\|_\infty = \max_{m \in \cM} |\alpha_m|.
\end{equation}

For a finite set $\cS$ we use $\#\cS$ to denote its cardinality. 
The letter $p$, with or without subscripts, always denotes a prime number.

We denote by $\F_p$ the finite field of $p$ elements, which we identify by $\{0,1,\ldots,p-1\}$, and in what follows, we mix the usage of $\F_p$ and $\{0,1,\ldots,p-1\}$. We freely alternate between the language of finite fields and the 
language of congruences.  

We also use $\N$ to denote that set of positive integers.

As usual, 
\begin{equation}\label{eq:tau}
\tau(k) = \#\{d \in \N:~ d \mid k\}
\end{equation}
denotes  
the divisor function, and we repeatedly use the classical bound
$\tau(k) = k^{o(1)}$
as $k \to +\infty$ (see, for example, in~\cite[Equation~(1.81)]{IK04}). 
For an integer $a$ coprime to $m$, 
$\overline{a}$ denotes the multiplicative inverse of $a\bmod m$, that is, $a\overline{a}\equiv1\bmod m$, which should not be confused with the complex conjugate.  

Given a function $f\in L^1(\R)$, that is, with a bounded $L^1$-norm over $\R$, 
 the Fourier transform is defined by
$$
\widehat{f}(y)=\int_\R f(x)\ue(-yx)\ud x
$$
with $\ue(z)=\exp(2\pi iz)$.

\section{Main results}
\label{sec:main}

\subsection{Bounds of multilinear character sums}
Put
\begin{equation}\label{eq:L-notation}
\mathscr{L}_1=(|a|K + |b|MN) \mand \mathscr{L}_2=(|a|KL + |b|MN).
\end{equation}
 We prove two estimates for the sums $\fT(\balpha,\bbeta)$ given by~\eqref{eq:trilinearcharactersum}, according to whether the weights $\alpha_m$ are identically $1$, which we write as $\balpha\equiv \mathbf{1}$, or arbitrary.

\begin{theorem}\label{thm:TypeI}
Let $K,M,N>1$ and $p>\max\{K,M,N\}$ a large prime. 
Uniformly over the weights $\balpha\equiv \mathbf{1}$, $\bbeta=(\beta_{k,n})$ and integers $a,b$ with $\gcd(ab,p)=1$, we have
$$
|\fT(\balpha,\bbeta)|
\leqslant\|\bbeta\|_\infty KMN\Delta_1\,p^{o(1)}
$$
for each positive integer $r$, provided that $M>4p^{\frac{1}{2r}}$, where
$$
\Delta_1=(1+NK^{-1})^{\frac{1}{2r}}(1+\mathscr{L}_1MNp^{-1-\frac{1}{2r}})^{\frac{1}{2r}}\(\frac{p^{1+\frac{1}{r}}}{M^4N^4}\)^{\frac{1}{4r}}
$$
with $\mathscr{L}_1$ as defined by~\eqref{eq:L-notation}. 
\end{theorem}

\begin{remark}\label{remark1}
Taking $N=K=1$ in Theorem~\ref{thm:TypeI}, we recover the celebrated Burgess bound for short linear character sums, 
which shows oscillations of non-trivial multiplicative characters modulo $p$, in any interval with length 
at least $p^{\frac{1}{4}+\varepsilon}$ for any $\varepsilon > 0$ and sufficiently large $p$. This is also known as the {\it Burgess threshold} (see~\cite[Theorem~12.26]{IK04}). 
\end{remark}

\begin{remark} 
\label{rem:1/8-treshold}
It is not easy to give a full description on the range of $(K,M,N)$ which is equivalent to $\Delta_1<1$. However,  when $a=b=1$, we note  that 
Theorem~\ref{thm:TypeI} is non-trivial if either of the following conditions
\begin{itemize}
\item $p^{\frac{1}{4}+\eta}<MN\leqslant p^{\frac{1}{2}}$ and  $N\leqslant K\leqslant p/(MN)$,
\item $MN\leqslant p^{\frac{1}{2}+\eta}\leqslant M^2NK$ and $N\ge K$,
\end{itemize}
holds with some fixed $\eta>0$.
For example, we can take 
\begin{equation}\label{eq:1/8}
K,M,N \sim p^{\frac{1}{8}+\eta},
\end{equation}
for small $\eta>0$.  The lower bound for the values of $K,M,N$ given by~\eqref{eq:1/8},
is the square root of the Burgess threshold. 
\end{remark}

We now use Theorem~\ref{thm:TypeI} to get a new bound, with a power saving, 
on the sums $\fQ_2(\bgamma)$ 
from Section~\ref{sec:surv T and Q}. Indeed, applying Theorem~\ref{thm:TypeI} with
$$
(m,n,k;a,b)\leftarrow (v,u,k;1,-1),\qquad (K,M,N)\leftarrow (x^2,x,x)
$$
and
$$
\beta_{n,k}=\mathop{\sum\sum\sum}_{\substack{1 \le r,s,u\le x\\ rs=k, \, u = n}}\gamma_{r,s,u},
$$
we find the following.

\begin{corollary}
\label{cor; Appl tp Q2}
Let $p$ be a large prime with $p > x^2 \ge 1$. 
Uniformly over the weights $\bgamma=(\gamma_{r,s,u})$, we have
$$
|\fQ_2(\bgamma)|\le \|\bgamma\|_\infty x^{4-\frac{2}{r}}(1+x^4p^{-1-\frac{1}{2r}})^{\frac{1}{2r}}p^{\frac{1}{4r}+\frac{1}{4r^2}+o(1)}
$$
for each positive integer $r$.
\end{corollary}

We see that Corollary~\ref{cor; Appl tp Q2}  gives a non-trivial bound for $\fQ_2(\bgamma)$ with a power-saving as long as $x\ge p^{\frac{1}{8}+\eta}$ with some fixed $\eta > 0$. And this also recovers the above bound for $\fQ_1(\balpha,\bbeta)$
with more general weights.

For arbitrary weights $\boldsymbol\alpha$ and $\bbeta$, we have the following alternative bound.  

\begin{theorem}\label{thm:TypeII}
Let $K,M,N>1$ and $p>\max\{K,M,N\}$ a large prime. 
Uniformly over the weights $\balpha=(\alpha_m)$, $\bbeta=(\beta_{k,n})$ and integers $a,b$ with 
$\gcd(ab,p)=1$, we have
$$
|\fT(\balpha,\bbeta)|
\leqslant\|\boldsymbol\alpha\|_\infty\|\boldsymbol\beta\|_2M(NK)^{\frac{1}{2}}\Delta_2\, p^{o(1)}
$$
for each positive integer $r\geqslant2$, provided that $N>4p^{\frac{1}{r}}$, where
$$
\Delta_2=(1+KM^{-2})^{\frac{1}{4r}}(1+\mathscr{L}_1MNp^{-1-\frac{1}{r}})^{\frac{1}{2r}}\(\frac{p^{1+\frac{1}{r}}}{K(MN)^2}\)^{\frac{1}{4r}}+M^{-\frac{1}{2}}
$$
with $\mathscr{L}_1$ as defined by~\eqref{eq:L-notation}.
\end{theorem}

\begin{remark} 
Fix $a=b=1$, then
Theorem~\ref{thm:TypeII} is non-trivial as long as
$$
M>p^\eta,\qquad  K(MN)^2>(1+KM^{-2})p^{1+\eta}, \qquad MN(K+MN) <p,
$$
holds with some fixed $\eta > 0$.  
For example, we can take
$$
K,M,N \sim p^{\frac{1}{5}+\eta},
$$
which has to be compared with~\eqref{eq:1/8} and also with the Burgess threshold.
\end{remark}

As an application of the above bounds for $\fT(\balpha,\bbeta)$, one is allowed to
address a modulo $p$ version of a question of Iwaniec and S{\'a}rk{\"o}zy~\cite{IS87} about distances between product sets and squares. 
We present such an application in \S~\ref{sec:applications}.

One can derive some bounds on such sums from Theorems~\ref{thm:TypeI} and~\ref{thm:TypeII}
with a trivial summation over $a$, which plays the role of $\ell$ in $\fQ(\boldsymbol\alpha,\boldsymbol\beta)$. However we may obtain a more 
precise bound as follows.

\begin{theorem}\label{thm:quadrilinearform} 
Let $K,L,M,N>1$ and $p>\max\{K,L,M,N\}$ a large prime. 
Uniformly over the weights $\balpha=(\alpha_{\ell,m})$
and  $\bbeta=(\beta_{k,n})$, we have
$$
|\fQ(\boldsymbol\alpha,\boldsymbol\beta)|
\leqslant \|\boldsymbol\alpha\|_\infty\|\boldsymbol\beta\|_2LM(KN)^{\frac{1}{2}}p^{o(1)}\cdot \Delta_3,
$$ 
for each positive integer $r\geqslant2$, provided that $K,N>4p^{\frac{1}{r}}$, where
\begin{align*}
\Delta_3
=(KLMN)^{-\frac{3}{4r}}&\(\frac{M}{K}+1\)^{\frac{1}{2r}}\(1+\frac{\mathscr{L}_2^2}{p^{1+\frac{1}{r}}}\)^{\frac{1}{2r}}\mathscr{L}_2^{\frac{1}{2r}}p^{\frac{1}{4r}+\frac{1}{2r^2}}\\
& \qquad  +(KN)^{-\frac{1}{2}}p^{\frac{1}{2r}}+(MN)^{-\frac{1}{2}}p^{\frac{1}{2r}}+(LM)^{-\frac{1}{2}}+p^{-\frac{1}{2}}
\end{align*}
with $\mathscr{L}_2$ as defined by~\eqref{eq:L-notation}.
\end{theorem}

We now analyse when Theorem~\ref{thm:quadrilinearform} wins over the trivial bound 
$$
|\fQ(\boldsymbol\alpha,\boldsymbol\beta)| \le \|\boldsymbol\alpha\|_\infty\|\boldsymbol\beta\|_2(LM)(KN)^{\frac{1}{2}}.
$$

\begin{corollary}  \label{cor:nontriv}  
Fix $a=b=1$.
For any  $\eta>0$ there is some $\kappa> 0$ such that if
\begin{equation}
\begin{cases}\label{manycases}
KL\ll MN\leqslant KLp^{\frac{1}{3}-\eta},\\ K^5L^3N \geqslant Mp^{1+\eta}, \\ (M/K)^5(N/L)^3\leqslant p^{1-\eta},\\
(KL)^3  MN\geqslant p^{1+\eta},\\
KN\geqslant p^\eta,\\
MN\geqslant p^\eta,\\
LM\geqslant p^\eta,
\end{cases}
\end{equation}
then
$$
|\fQ(\boldsymbol\alpha,\boldsymbol\beta)|
\leqslant \|\boldsymbol\alpha\|_\infty\|\boldsymbol\beta\| LM (KN)^{\frac{1}{2}}p^{-\kappa}.
$$
\end{corollary}

To see this,  we write the quantity $\Delta_3$ in an obvious manner as
\begin{equation}\label{eq:Delta3-Delta3'}
\Delta_3=\widetilde \Delta_3+(KN)^{-\frac{1}{2}}p^{\frac{1}{2r}}+(MN)^{-\frac{1}{2}}p^{\frac{1}{2r}}+(LM)^{-\frac{1}{2}}+p^{-\frac{1}{2}}
\end{equation}
corresponding to five terms in its definition. By taking $r$ sufficiently large, each of the last four terms in~\eqref{eq:Delta3-Delta3'} is at most $p^{-\kappa}$ for some $\kappa>0$, provided that the last three conditions in~\eqref{manycases} hold.
It remains to check that $\widetilde \Delta_3\ll p^{-\kappa}$. We now assume $KL\ll MN$ so that $\mathscr{L}_2=KL+MN\ll  MN$.
First, rising  $\widetilde \Delta_3$ to the power $2r$ and  expanding it, we obtain four terms. Each of these terms is at most $p^{-\kappa}$ as long as $KL\ll MN$ and
$$
\max\left\{(MN)^2 p^{-\frac 12}, \, \dfrac MKp^{\frac 12 +\frac 1r},\, \dfrac {M^3N^2}{Kp^{\frac 12}} ,\, 
p^{\frac 12 +\frac 1r}\right\}  \ll (KL)^{\frac 32} (MN)^{\frac 12}p^{-2r\kappa}.
$$
These leads to the remaining inequalities in~\eqref{manycases}, after  choosing $r$ to be large enough.

In particular,  we see from Corollary~\ref{cor:nontriv} that we have a non-trivial 
bound on $\fQ(\boldsymbol\alpha,\boldsymbol\beta)$ provided that 
$$
K=L=M=N>p^{\frac{1}{8}+\eta}.
$$
Observe that one more time the exponent $1/8$ appears as a threshold, 
see Remark~\ref{rem:1/8-treshold}, 
however this does not seem to follow from Theorem~\ref{thm:TypeI}.

\subsection{Consequences and applications} 
\label{sec:applications}
Theorems~\ref{thm:TypeI} and~\ref{thm:TypeII} describe general situations since the sequence $\bbeta$ satisfies no particular hypotheses. For instance, if  we choose $\balpha$ and $\bbeta$ in the landscape of multiplicative or additive characters modulo $p$, we obtain, as 
a direct consequence of Theorem~\ref{thm:TypeII},  the following bound for triple character sums on short initial segments.
To proceed, we use $\psi$ to denote a non-trivial additive character of $\F_p$.

\begin{corollary}
\label{cor:trilin rat func}
Fix a rational function $P$ in $\F_p(X)$ and a rational function $Q$ in $\F_p(X,Y)$.
For every fixed $\eta >0$ there exists some $\kappa>0$,   such that
$$
\sum_{k\leqslant K}\sum_{m\leqslant M}\sum_{n\leqslant N}\chi(k+mn)\psi(P(m)+Q(k,n))
\ll KMN p^{-\kappa}
$$
and
$$
\sum_{k\leqslant K}\sum_{m\leqslant M}\sum_{n\leqslant N}\chi((k+mn)P(m)Q(k,n))
\ll KMN p^{-\kappa}
$$
hold uniformly for
$$ p^{\frac 15 + \eta}\leq K, \, M, \, N <p^{\frac 14 -\eta}.
$$
\end{corollary}

\begin{remark}
Theorem~\ref{thm:quadrilinearform} easily leads to an obvious quadrilinear version 
of Corollary~\ref{cor:trilin rat func},  but with $K,L,M$ and $N$ of size at least $p^{1/8+\eta}$
for any $\eta> 0$. 
\end{remark}

Furthermore, Iwaniec and  S{\'a}rk{\"o}zy~\cite{IS87} have considered the following multiplicative hybrid problem  
with positive integers: 
given two arbitrary subsets $\cM$ and $\cN$, how close is the product $mn$ to a   
square in $\Z$ with $(m,n)\in\cM\times\cN$? 
A special case of their general result asserts that
for any $\cM,\cN\subseteq[N,2N]$ with $\#\cM,\#\cN\gg N$, there exist $(m,n,\ell )\in\cM\times\cN\times\N$ satisfying
$$\frac{mn-\ell^2}{\ell}\ll (N/\log N)^{-\frac{1}{2}}.$$
Iwaniec and S{\'a}rk{\"o}zy~\cite{IS87} have also conjectured that the upper bound might be replaced by $N^{-1+o(1)}$. 

Theorem~\ref{thm:TypeII} allows us to study a modular analogue of the above result of 
Iwaniec and  S{\'a}rk{\"o}zy, for which the above conjectural bound $N^{-1+o(1)}$ can be realised in some particular cases. To set up, we consider subsets $\cM,\cN\subseteq[1,p]$, we examine the distance between $mn$, with $(m,n)\in\cM\times\cN$, and quadratic residues modulo $p$.

\begin{theorem}\label{thm:IS-analogue}
Fix a positive integer $r$, two real numbers $0<c_0<1$ and $\eta>0$. Then there exists a constant $P$, depending only on $(r,c_0,\eta)$, such that
\begin{itemize}
\item  for every prime $p\geqslant P$,
\item  for every $M$ and $N$ satisfying  
 \begin{equation}\label{condforMandN}
M^4N^2\geqslant p^{1+\frac{1}{r}+\eta},
\end{equation}
\item  for every subsets $\cM,\cN\subseteq[1,p]$ with 
\begin{equation}\label{lowerdensity}
\cM\subseteq [M,2M], \quad \cN\subseteq [N,2N], \quad \#\cM\geqslant c_0M,\quad \#\cN\geqslant c_0 N,
\end{equation}
\end{itemize}
there exist  $m\in \cM$, $n\in \cN$, and some positive integer $k$ satisfying 
$$
k \le p^{1+\frac{1}{r}+\eta}(MN)^{-2}+p^\eta
$$
such that $mn+k$ is a quadratic residue modulo $p$.
\end{theorem}

Actually, the proof which is given in \S~\ref{sec:proofofIS-analogue}  produces a lower bound for the cardinality of the set of triples  we are interested in 
\begin{align*}
\#\{(m,n,k):~m \in \cM,\, n &\in \cN, \, k \leq K    \text{ and } mn+k \\ 
& \text{ is a quadratic residue}\bmod p\}  \gg KMN,
\end{align*}
where $K\sim p^{1+\frac{1}{r}+\eta}(MN)^{-2}$.

To illustrate Theorem~\ref{thm:IS-analogue}, choose  $M=N\sim p^{\frac{1}{5}+\delta}$ (with $\delta >0$ very small)  and  $\cM$ and $\cN $ subsets of $[N, 2N]$
such that $\#\cM,\#\cN \geqslant \delta N$. Then, for sufficiently large $p$,  there exists a positive integer $k\leq p^{\frac{1}{5}}$ 
and $(m,n)\in\cM\times\cN$ such that $mn + k$ is a quadratic residue modulo $p$.
Recall that there must exist a square of an integer between $t$ and $t+O (\sqrt t)$ for all $t>1$ (namely, $\lceil \sqrt t\rceil^2$ is such a square) and that essentially, nothing better than the upper bound  $p^{\frac{1}{4}+o(1)}$ is proved for the distance between two consecutive quadratic residues modulo $p$. The above illustration shows that one can do much better if squares are replaced by quadratic residues modulo $p$.

The proof of  Theorem~\ref{thm:IS-analogue} can be easily generalized in many aspects.
For instance one can relax the hypothesis~\eqref{lowerdensity} about the cardinalities of $\cM$ and $\cN$ by choosing 
\begin{equation}\label{defalphabeta}
 \begin{cases}
 \alpha_m = {\bf 1}_{[M,2M]\cap \cP}(m),\\
 \beta_{k,n} = {\bf 1}_{[K,2K]\cap \cP}(k)\cdot {\bf 1}_{[N,2N]\cap \cP}(n),
 \end{cases}
 \end{equation}
where $\cP$ denotes the set of all primes, and ${\bf 1}_{S}$ is the indicator function of a set $S$.
One then deduces from  Theorem~\ref{thm:TypeII} the following result by counting primes using Chebyshev's inequalities.
 
\begin{corollary} \label{CHEN} 
Fix $\eta>0$. For every sufficiently large $p$, there exist three primes $p_1,p_2$ and $p_3$ such that 
\begin{itemize}
\item    $p_1, p_2, p_3 \sim p^{{\frac 15}+\eta}$,
\item  $p_1p_2+p_3 $ is a quadratic non-residue modulo $p$.
\end{itemize}
\end{corollary}

As we have mentioned,  Corollary~\ref{CHEN} follows directly from  Theorem~\ref{thm:TypeII} 
and it can be modified in various ways, for example one can impose different arithmetic restrictions 
on the primes $p_1, p_2, p_3$. For example, one further imposes
that the three shifted primes
$p_1+2, p_2+2, p_3 +2$ have at most two prime factors. This statement is deduced from the famous work of Chen~\cite{Ch73} about the twin prime conjecture that we write under the form of the following inequality
$$
\#\{ p\leqslant x :~p+2 \text{ has at most two prime factors}\} \gg x(\log 2x)^{-2},
$$
for $x\geqslant 2$.
One can also appeal to~\cite[Corollaire~2]{Fo82} to introduce a more involved definition of the coefficient $\beta_{k,n}$ (see~\eqref{defalphabeta}) leading to a new version of  Corollary~\ref{CHEN}, where we impose on $p_1p_3+2$ to have at most two prime factors.

Similarly, Theorem~\ref{thm:quadrilinearform}  leads to the following result (which can also be 
modified along the lines mentioned in the above).

\begin{corollary} \label{cor: p1234} 
Fix $\eta>0$. For every sufficiently large $p$, there exist four primes $p_1,p_2,p_3$ and $p_4$ such that 
\begin{itemize}
\item    $p_1, p_2, p_3, p_4 \sim p^{{\frac 18}+\eta}$,
\item  $p_1p_2+p_3 p_4$ is a quadratic non-residue modulo $p$.
\end{itemize}
\end{corollary}

Another application of  Theorem~\ref{thm:quadrilinearform} concerns sums with the divisor 
function (see~\eqref{eq:tau} above)
$$
S(U,V) = \sum_{u\leqslant U} \sum_{v\leqslant V} \tau(u) \tau(v) \chi(u-v).
$$
These sums are two-dimensional analogues of the sum
$$
S_a(U) =  \sum_{u\leqslant U} \tau(u)  \chi(u+a),
$$
which has been studied in a number of works~\cite{BGKS13, Ch10, Ka71,Ka00}.  
In particular, it is shown in~\cite[Theorem~27]{BGKS13} that for any $\eta> 0$ there is some $\kappa > 0$  
such that for $U > p^{\frac{1}{3}+\eta}$, uniformly over integers $a$ with $\gcd(a,p)=1$,  one has
$$
S_a(U)  \ll Up^{-\kappa}.
$$
It is easy to see that Theorem~\ref{thm:quadrilinearform} combined with the 
standard completing technique (see~\cite[Section~12.2]{IK04}) implies the following result.

\begin{corollary} \label{cor:div sum} 
For any fixed $\eta> 0$, there is some $\kappa > 0$  
such that for $U, V > p^{\frac{1}{4}+\eta}$
we have
$$
S(U,V)  \ll UVp^{-\kappa}.
$$
\end{corollary}

One can also use  Theorem~\ref{thm:quadrilinearform}  to estimate a rich 
variety of other quantities, for example, sums over primes
$$
W(x) = \sum_{p_1, p_2, p_3, p_4\le x} \(\frac{p_1}{p_3}\) \(\frac{p_2}{p_4}\)\chi \(p_1p_2-p_3p_4\)
$$
with weights given by Legendre symbols. Theorem~\ref{thm:quadrilinearform} implies a 
bound on $W(x)$ with a power saving, provided $x \ge p^{\frac{1}{8} + \eta}$ for  any fixed $\eta> 0$.We also note a variety of bounds  on characters over various arithmetic sequences 
can be found in a very informative survey of Karatsuba~\cite{Ka08}.

In closing this section, we would like to state a striking application of our general bounds in 
Theorem~\ref{thm:quadrilinearform} to character sums involving Farey fractions.
For $R\geqslant2$ let 
$$
\cF(R)= \{ r/s:~\gcd(r,s) = 1, \ 0\le r\leqslant s \le R\}
$$ 
be the set of Farey fractions of order $R$. This can be embedded in $\F_p$ in a canonical way
$r/s \mapsto r \overline s \pmod p$, where $ \overline s $ is the multiplicative inverse of $s$ modulo $p$ (which is well-defined for $R< p$ and in fact is injective for $R < p^{\frac{1}{2}}$).

By virtue of the multiplicativity of $\chi$, in the form
$$
\alpha_{\ell, m} \beta_{k, n}\chi (\ell \overline{m} -\overline{ k}n)=
\left( \alpha_{\ell, m} \overline{\chi} (m)\right) \left( \beta_{k, n}\overline{ \chi }(k) \right) \chi (k\ell -mn)
$$
 we may derive the following consequence directly from 
Theorem~\ref{thm:quadrilinearform}.

\begin{corollary} \label{cor:Farey} 
Let $R\geqslant2$ and $\xi_\rho, \zeta_\varrho$ be bounded weights
supported on $\cF(R)$. Then for any fixed $\eta> 0$, there is some $\kappa > 0$  
such that for $p^{\frac{1}{8}+\eta}\leqslant R<p^{\frac{1}{2}}$ we have
$$
\mathop{\sum\sum}_{\rho, \varrho \in \cF(R)} \xi_\rho \zeta_\varrho\chi(\rho - \varrho) \ll R^4 p^{-\kappa}.
$$ 
\end{corollary}

Note that the use of bilinear bounds~\eqref{eq:PV-threshold} and~\eqref{eq:Karatsuba-implicit} would lead to a 
much more restrictive 
condition $p^{\frac{1}{4} + \eta}\leqslant R<p^{\frac{1}{2}}$ in Corollary~\ref{cor:Farey}.

\section{Preliminaries}

\subsection{Moments of some character sums} 
The following result  is a consequence of Riemann Hypothesis for curves over finite fields due to Weil~\cite{We48}.
 It is a slight extension of a classical bound due to 
Davenport and Erd\H{o}s~\cite{DE52},  which corresponds to the case when  $\cA$ is an interval and $\boldsymbol\gamma=1$. The proof of Lemma~\ref{lm:DavenportErdos} below is identical, and is 
also a part of the arguments in Appendix~\ref{appendix:Karatsuba}.

\begin{lemma}\label{lm:DavenportErdos}
Let $\cA\subseteq\F_p$ be a subset of cardinality $A$, and $\chi$ a non-trivial multiplicative character of $\F_p^\times$. For each positive integer $r$ and any complex-valued weight $\boldsymbol\gamma=(\gamma_a)$ with $\|\boldsymbol\gamma\|_\infty\leqslant1$, we have
$$
\sum_{x\in\F_p}\left| \sum_{a\in\cA}\gamma_a\chi(x+a)\right|^{2r}\leqslant (2r)^r(A^rp+A^{2r}p^{\frac{1}{2}}).
$$
\end{lemma}

We also need the following version of Lemma~\ref{lm:DavenportErdos}, which again essentially
repeats the argument in~\cite{DE52} but uses {\it Weil's bound} for multiplicative 
character sums with polynomials; see, for example,~\cite[Theorem~11.23]{IK04}. 

\begin{lemma}\label{lm:DavenportErdos2}
Let $\cA\subseteq\F_p$ be a subset of cardinality $A$,  and $\chi_1,\chi_2$ two non-trivial multiplicative characters of $\F_p^\times$. For each positive integer $r$ and any complex-valued weight $\boldsymbol\gamma=(\gamma_a)$ with $\|\boldsymbol\gamma\|_\infty\leqslant1$, we have
$$
\mathop{\sum\sum}_{x,y\in\F_p}\left| \sum_{a\in\cA}\gamma_a\chi_1(x+a)\chi_2(y+a)\right| ^{2r}\leqslant (2r)^r(A^rp^2+2rA^{2r}p).
$$
\end{lemma}

\begin{proof}
Denote by $S$ the quantity in question. Opening the power, we write
$$
S
\leqslant\mathop{\sum\sum}_{\ba,\bb\in\cA^r}|S(\ba,\bb;\chi_1)S(\ba,\bb;\chi_2)|,
$$
where for $\ba=(a_1,\cdots,a_r)\in\cA^r$, $\bb=(b_1,\cdots,b_r)\in\cA^r$ and each non-trivial multiplicative character $\chi$ of $\F_p^\times$,
$$
S(\ba,\bb;\chi)=\sum_{x\in\F_p}\prod_{1\leqslant j\leqslant r}\chi(x+a_j)\overline{\chi(x+b_j)}.
$$
The subsequent treatment is uniform in all non-trivial multiplicative characters $\chi$ of $\F_p^\times$.
If the coordinates of $\ba,\bb$ appear in pairs (with necessary permutations), we appeal to the trivial bound
$$
|S(\ba,\bb;\chi)|\leqslant p.
$$
Note that the number of such tuples $(\ba,\bb)$ is at most $r\binom{2r}{r}A^r\leqslant (2rA)^r$. For the remaining $\ba,\bb$, we can apply Weil's bound for complete character sums (see~\cite[Corollary~11.24]{IK04}), getting
$$
|S(\ba,\bb;\chi)|\leqslant 2rp^{\frac{1}{2}}.
$$
This completes the proof of Lemma~\ref{lm:DavenportErdos2} by taking all possibilities of $\ba,\bb$ into account.
\end{proof}

\subsection{Bounds of some  GCD sums}

We need the following estimate.

\begin{lemma}\label{lm:gcdsum}
Let $a,b$ be non-zero integers. Let $A,B,K,L,M,N,U,W\geqslant1$ with $A\ll  |a|KL  + |b|MN$ and $B\ll LU+ MW$. Then we have
\begin{multline*}
\mathop{\sum_{k\leqslant K}\mathop{\sum\sum}_{\ell_1,\ell_2\leqslant L}\mathop{\sum\sum}_{m_1,m_2\leqslant M}\sum_{n\leqslant N}\sum_{u\sim U}\sum_{w\sim W}}_{\substack{ak\ell_\nu + bm_\nu n=a_\nu  ,~a\ell_\nu u+bm_\nu w=b_\nu, \ \nu =1,2 \\ 
|a_1|\sim A,~|b_2|\sim B,~kw\neq nu}}\gcd(a_1,b_1)\gcd(a_2,b_2)
\leqslant ABLM(KW+NU)\Upsilon^{o(1)}
\end{multline*}
with
$\Upsilon=|ab|ABKLMNUW$.
\end{lemma}

\begin{proof}
Denote by $G$ the above quantity in question.
Writing $g_1 = \gcd(a_1, b_1)$ and $g_2 = \gcd(a_2,b_2)$ we have the inequality 
\begin{align*}
G
&\leqslant \mathop{\sum\sum}_{g_1\ll A,\, g_2\ll B}g_1g_2
\mathop{\sum_{k\leqslant K}\mathop{\sum\sum}_{\ell_1,\ell_2\leqslant L}\mathop{\sum\sum}_{m_1,m_2\leqslant M}\sum_{n\leqslant N}\sum_{u\sim U}\sum_{w\sim W}}_{\substack{ak\ell_\nu  +b m_\nu  n\equiv a\ell_\nu  u + bm_\nu  w \equiv0\bmod{g_\nu}, \ \nu =1,2 \\
|ak\ell_1+bm_1n|\sim A,~|a\ell_2u+bm_2w|\sim B,~kw\neq nu}}1.
\end{align*}
The  congruences in the summation  imply
$$g_\nu \mid m_\nu w (ak\ell_\nu +bm_\nu n)-m_\nu n (a\ell_\nu u + bm_\nu w) =a\ell_\nu m_\nu(kw - nu)$$
for $\nu =1,2$,  so that we can decompose $g_\nu$ (in  a not  necessarily unique way)   as $$g_\nu = d_\nu e_\nu f_\nu$$ where 
$$
d_\nu \mid a\ell_\nu, \qquad e_\nu \mid  m_\nu,  \qquad f_\nu \mid (kw-nu).
$$
In particular, we have 
$$
\lcm[f_1,f_2]\le 2(KW+NU).
$$

Therefore, we have   
$$
G\leqslant \mathop{\sum\sum}_{\substack{d_1e_1f_1\ll A\\ d_2e_2f_2\ll B\\ \lcm[f_1,f_2]\le 2(KW+NU)}}d_1d_2e_1e_2f_1f_2\mathop{\sum_{k\leqslant K} \sum_{n\leqslant N}\sum_{u\sim U}\sum_{w\sim W}}_{\lcm[f_1,f_2]\mid (kw-nu)\neq0}\mathop{\mathop{\sum\sum\sum\sum}_{\ell_1,\ell_2\leqslant L,~m_1,m_2\leqslant M}}_{\substack{ak\ell_1 + bm_1 n
\equiv 0\bmod{d_1 e_1 f_1 } \\
a\ell_2u+bm_2w\equiv0\bmod{d_2e_2f_2}\\
\ell_\nu \equiv0\bmod{d_\nu },\, m_\nu \equiv0\bmod{e_\nu },\ \nu =1,2 \\
|ak\ell_1+ bm_1n|\sim A,~|a\ell_2u+bm_2w|\sim B}}1.
$$
Making the change of variables $\ell_\nu \to d_\nu\ell_\nu$,  $m_\nu \to e_\nu m_\nu$, $\nu =1,2$, 
we obtain 
$$
G\leqslant \mathop{\sum\sum}_{\substack{d_1e_1f_1\ll A\\ d_2e_2f_2\ll B\\ \lcm[f_1,f_2]\le 2(KW+NU)}}d_1d_2e_1e_2f_1f_2
\mathop{\sum_{k\leqslant K}\sum_{k\leqslant N}\sum_{u\sim U}\sum_{w\sim W}}_{\lcm[f_1,f_2]\mid (kw-nu)\neq0}
\mathop{\mathop{\sum\sum\sum\sum}_{d_1\ell_1,\,d_2\ell_2\leqslant L,~e_1m_1,\,e_2m_2\leqslant M}}_{\substack{ad_1\ell_1k +be_1m_1 n \equiv0\bmod{d_1 e_1 f_1}  \\
ad_2 \ell_2u+be_2 m_2w\equiv0\bmod{d_2e_2f_2}\\
|ad_1\ell_1k+be_1m_1 n |\sim A,~|ad_2 \ell_2u+be_2m_2w|\sim B}}1.
$$ 
Note that  the congruences 
$$
ad_1 \ell_1k +be_1m_1 n \equiv0\bmod{d_1 e_1 f_1} \mand 
ad_2 \ell_2u+be_2 m_2w\equiv0\bmod{d_2e_2f_2}
$$ 
imply 
$$
e_1 \mid ad_1\ell_1k, \qquad d_2\mid be_2 m_2 w,
$$
and
$$
bm_1n\equiv-ad_1\ell_1 k/e_1\bmod{d_1f_1} ,\qquad
a\ell_2u\equiv -be_2 m_2w/d_2\bmod{e_2f_2},
$$
respectively. 
It then follows that
\begin{align*}
G
&\leqslant \mathop{\sum\sum}_{\substack{d_1e_1f_1\ll A\\ d_2e_2f_2\ll B\\ \lcm[f_1,f_2]\le 2(KW+NU)}}d_1d_2e_1e_2f_1f_2\mathop{\sum_{k\leqslant K} \sum_{n\leqslant N}\sum_{u\sim U}\sum_{w\sim W}}_{\lcm[f_1,f_2]\mid (kw-nu)\neq0}
\mathop{\mathop{\sum\sum\sum\sum}_{d_1\ell_1,d_2\ell_2\leqslant L,~e_1m_1,e_2m_2\leqslant M}}_{\substack{e_1\mid ad_1\ell_1k,~bm_1n\equiv-ad_1\ell_1k/e_1\bmod{d_1f_1}\\d_2\mid be_2m_2w,~a\ell_2u\equiv-be_2m_2w /d_2\bmod{e_2f_2}\\
|ad_1\ell_1k +be_1m_1 n |\sim A,~|ad_2 \ell_2 u+be_2m_2w|\sim B}}1\\
&\leqslant \mathop{\sum\sum}_{\substack{d_1f_1\ll A\\ e_2f_2\ll B\\ \lcm[f_1,f_2]\le  2(KW+NU)}} d_1e_2f_1f_2\sup_{\delta\in\Z}\mathop{\sum_{k\leqslant K} \sum_{w\sim W}}_{kw\equiv \delta\bmod{\lcm[f_1,f_2]}}\mathop{\sum\sum}_{\substack{d_2\ll B/(e_2f_2)\\ e_1\ll A/(d_1f_1)}}d_2e_1\\
&\qquad\qquad\qquad\qquad\qquad\qquad\quad\times
\sum_{n\leqslant N}\sum_{u\sim U}\mathop{\mathop{\sum\sum\sum\sum}_{d_1\ell_1,d_2\ell_2\leqslant L,~e_1m_1,e_2m_2\leqslant M}}
_{\substack{e_1\mid ad_1\ell_1k,~bm_1n\equiv-ad_1\ell_1k/e_1\bmod{d_1f_1}\\d_2\mid be_2m_2w,~a\ell_2u\equiv-be_2m_2w /d_2\bmod{e_2f_2}\\
|ad_1\ell_1k/e_1 +bm_1 n |\sim A/e_1,~|a\ell_2 u+be_2m_2w/d_2|\sim B/d_2}}1.
\end{align*}

To proceed, we group the variables $m_1,n$, and $\ell_2,u$, separately, so that we need to count the number of tuples $(r,s)$ with
$$
|ad_1\ell_1k/e_1+br|\sim A/e_1,~|as+be_2m_2w/d_2|\sim B/d_2
$$
and
$$
br\equiv-d_1\ell_1k/e_1\bmod{d_1f_1},\qquad as\equiv-e_2m_2w /d_2\bmod{e_2f_2}.
$$
 It is clear that the number $T$ of such tuples $(r,s)$ satisfies
$$
T \ll \(1+\frac{A\gcd(b,d_1f_1)}{bd_1e_1f_1}\)\(1+\frac{B\gcd(a,e_2f_2)}{ad_2e_2f_2}\)
\ll \frac{AB}{d_1d_2e_1e_2f_1f_2}.
$$
This leads us to
$$
G
\ll AB\mathop{\sum\sum}_{\substack{d_1f_1\ll A\\ e_2f_2\ll B\\ \lcm[f_1,f_2]\le 2(KW+NU)}} \sup_{\delta\in\Z}\mathop{\sum_{k\leqslant K} \sum_{w\sim W}}_{kw\equiv \delta\bmod{\lcm[f_1,f_2]}}\mathop{\sum\sum}_{\substack{d_2\ll B/(e_2f_2)\\ e_1\ll A/(d_1f_1)}}\mathop{\sum\sum}_{\substack{\ell_1\leqslant L/d_1,~m_2\leqslant M/e_2\\d_2\mid be_2m_2w,~e_1\mid ad_1\ell_1k}}1.
$$
In what follows, we would like to sum over $d_2,e_1$ firstly, and then $\ell_1,m_2$, so that
$$
G
\le  ABLM\Upsilon^{o(1)}\mathop{\sum\sum}_{\substack{d_1f_1\ll A\\ e_2f_2\ll B\\ \lcm[f_1,f_2]\le 2(KW+NU)}} \frac{1}{d_1e_2}\sup_{\delta\in\Z}\mathop{\sum_{k\leqslant K} \sum_{w\sim W}}_{kw\equiv \delta\bmod{\lcm[f_1,f_2]}}1.
$$

Again by grouping variables, we arrive at
\begin{align*}
G
&\leqslant ABLM\Upsilon^{o(1)}\mathop{\sum\sum}_{\substack{d_1f_1\ll A\\ e_2f_2\ll B\\
 \lcm[f_1,f_2]\le 2(KW+NU)}}\frac{1}{d_1e_2}\(\frac{KW}{\lcm[f_1,f_2]}+1\)\\
&\leqslant ABLM\Upsilon^{o(1)}
 \sum_{d_1 \ll A} \frac{1}{d_1} \sum_{e_1 \ll A} \frac{1}{e_2}
\sum_{f \ll KW+NU}\(\frac{KW}{f}+1\)\\
&\leqslant ABLM(KW+NU)\Upsilon^{o(1)},
\end{align*}
where we have used the bound on the divisor function to count the number of pairs $(f_1,f_2)$ with $\lcm[f_1,f_2] = f$.
This completes the proof of the lemma.
\end{proof}

\section{Proof of Theorem~\ref{thm:TypeI}}
\label{sec:Type-I proof}
\subsection{Preparations}
Throughout this section, we assume that $\balpha$ is identically 1 on its support.
Suppose without loss of generality that 
$\|\bbeta\|_\infty \le 1$ and $M$ is a positive integer. Following an approach of Friedlander and Iwaniec~\cite{FI93},
we define the function $w\in L^1(\R)$ as
$$
w(x)=
\begin{cases}
\min\{x,~1,~M+1-x\},& \text{for~}x\in[0,M+1],\\
0,&\text{otherwise}.
\end{cases}
$$
The integration by parts implies
\begin{equation}\label{eq:w-fourier}
\widehat{w}(\xi)\ll \min\{M,|\xi|^{-1},|\xi|^{-2}\}.
\end{equation}

\subsection{Amplification}
The above function $w$ allows us to run the summation over $m$ to the whole set $\Z$.
That is, 
$$
\fT(\balpha,\bbeta)=\sum_{k\leqslant K} \sum_{m\in\Z}\sum_{n\leqslant N}w(m)\beta_{k,n}\chi(ak+bmn).
$$
Furthermore, for all integers $u,v$ we have
$$
\fT(\balpha,\bbeta)=\sum_{k\leqslant K} \sum_{m\in\Z}\sum_{n\leqslant N} w(m+uv)\beta_{k,n}\chi(ak+b(m+uv)n).
$$
We choose two real positive parameters $U$ and $V$ with 
\begin{equation}\label{eq:UV=M/4}
U,V\geqslant1,\ \  UV=\frac{1}{4}M.
\end{equation}
By Fourier inversion and summing over integers $u$ and $v$ with $u\sim U$, $v\sim V$  
we have the amplified expression
$$
\fT(\balpha,\bbeta)\ll\frac{1}{UV} \sum_{k\leqslant K} \sum_{m\leqslant M}  \sum_{n\leqslant N}\sum_{u\sim U}\int_\R |\widehat{w}(\xi)|\left| \sum_{v\sim V}\ue(uv\xi)\chi(\overline{un}(ak+bmn)+bv)\right|\ud\xi.
$$
Making the change of variable $\xi\rightarrow \xi/u$ and replacing $UV$ with $M$, 
in view of~\eqref{eq:w-fourier}, it follows that
\begin{equation}
\begin{split} 
\label{eq:T1-Transformed}
\fT(\balpha,\bbeta)
& \ll\frac{1}{M}\sum_{k\leqslant K}  \sum_{m\leqslant M}  \sum_{n\leqslant N}\sum_{u\sim U}\int_\R \left| \widehat{w}\(\frac{\xi}{u}\)\right|\left| \sum_{v\sim V}\ue(v\xi)\chi(\overline{un}(ak+bmn)+bv)\right|\frac{\ud\xi}{u}\\
&\ll\frac{1}{M}\int_\R 
\min\left\{\frac{M}{U},\frac{1}{|\xi|},\frac{U}{|\xi|^2}\right\} \\
& \qquad \qquad \qquad \times \sum_{k\leqslant K}\sum_{m\leqslant M}\sum_{n\leqslant N}\sum_{u\sim U}\left| \sum_{v\sim V}\ue(v\xi)\chi(\overline{un}(ak+bmn)+bv)\right|\ud\xi.
\end{split} 
\end{equation}
This yields
$$
\fT(\balpha,\bbeta)\ll\frac{\log M}{M} \sum_{k\leqslant K} \sum_{m\leqslant M}  \sum_{n\leqslant N} \sum_{u\sim U}\left|\sum_{v\sim V}\ue(v\xi)\chi(\overline{un}(ak+bmn)+bv)\right|
$$
for some $\xi\in\R$ (for which the multiple sum in~\eqref{eq:T1-Transformed}  
achieves its largest possible value). 

We may further restrict our summation to triples $(k,m,n)$ with $p\nmid(ak+bmn)$ up to an 
error term at most
$$
 p^{o(1)}\mathop{\sum_{k\leqslant K}\sum_{\ell  \leqslant  MN}}_{ak+b\ell\equiv0\bmod p}1
\leqslant \(\frac{MN}{p}+1\)Kp^{o(1)}.
$$
Therefore,
\begin{align*}
\fT(\balpha,\bbeta)\ll\frac{\log M}{M}\mathop{\sum_{k\leqslant K} \sum_{m\leqslant M}  \sum_{n\leqslant N}}_{p\nmid(ak+bmn)}\sum_{u\sim U}\left|\sum_{v\sim V}\ue(v\xi)\chi(\overline{un}(ak+bmn)+bv)\right| &\\ 
+\(\frac{MN}{p}+1\)&Kp^{o(1)}.
\end{align*}

\subsection{Regrouping, counting and Weil's  bound}
Put
$$
\varrho(x)=\mathop{\sum_{m\leqslant M}\sum_{n\leqslant N}\sum_{k\leqslant K}\sum_{u\sim U}}_{ak+bmn\equiv unx\not\equiv0\bmod p}1
$$
for $x\bmod p$. We now have
$$
\fT(\balpha,\bbeta)\ll\frac{\log M}{M}\sum_{x\bmod p}\varrho(x)\left| \sum_{v\sim V}\ue(v\xi)\chi(x+bv)\right|+\(\frac{MN}{p}+1\)Kp^{o(1)}.
$$
Applying the H\"older inequality with $r\ge 1$, we obtain 
\begin{equation}
\label{eq:T-I1I2S}
\fT(\balpha,\bbeta)\ll\frac{\log M}{M}\cI_1^{1-\frac{1}{r}}(\cI_2\cdot \cI_3)^{\frac{1}{2r}}+\(\frac{MN}{p}+1\)Kp^{o(1)}
\end{equation}
with
$$
\cI_j=\sum_{x\bmod p}\varrho(x)^j,\qquad  j=1,2,
$$
and
$$
\cI_3=\sum_{x\bmod p}\left| \sum_{v\sim V}\ue(v\xi)\chi(x+bv)\right|^{2r}.
$$

Trivially we have 
\begin{equation}\label{eq:I1-upperbound}
\cI_1\ll KMNU.
\end{equation}
Regarding $\cI_2$, we refer to the following lemma.

\begin{lemma}\label{lm:I2-upperbound}
With the above notation, we have
\begin{equation}\label{eq:I2-upperbound}
\cI_2
\leqslant  KMU(K+N)\(1+ \frac{\(|a|K + |b|MN\)UN}{p}\)p^{o(1)}.
\end{equation}  
\end{lemma}

\begin{proof} 
It is easy to see that  $\cI_2$ is  the number of 8--tuples 
$$
(k,\tk, m,\tm ,n,\tn ,u,\tu ) \in \N^8
$$ 
satisfying the conditions
\begin{align*}
 k,\tk\leqslant K, \quad m,&\, \tm \leqslant M, \quad n,\tn \leqslant N,\quad u,\tu \sim U,\\
(ak + bmn)\tn \tu &\equiv (a \tk + b \tm \tn)n u \not \equiv 0\pmod p, 
\end{align*}
which  is bounded by the number of 9--tuples 
$$(k,\tk, m,\tm ,n,\tn ,u,\tu ,t) \in \N^8 \times \Z, 
$$ 
satisfying
\begin{equation}
\label{eq:Diophantineequation}
\begin{split} 
 k,\tk\leqslant K, \quad m,\tm \leqslant M, &\quad n,\tn \leqslant N,\quad u,\tu \sim U,\quad  |t|\leqslant T,\\
(ak + bmn)\tn \tu &=(a \tk + b \tm \tn)nu+tp,\\
(ak + bmn)&(a \tk + b \tm \tn)\neq 0,
\end{split} 
\end{equation}
with 
\begin{equation}\label{eq:T-choice}
T=1+(|a|K + |b| MN )NU/p.
\end{equation}
Note that $\tn \mid a\tk un+tp$ and $a\tk un \not \equiv 0 \pmod p$. 
Hence $a\tk un+tp \ne 0$ for any $t \in \Z$ and 
thus, for given $\tk, n, t, u$, which we can fix in $O\(KNTU\)$ ways 
 the number of $\tn $ satisfying~\eqref{eq:Diophantineequation} is at most $p^{o(1)}$. After $\tk,\tm ,n,\tn , t,u$ are fixed, there are at most  $(K/N+1)p^{o(1)}$ possibilities for the tuple $(m,\tu ,k)$. This is due to the observations 
 that $\tu \mid (a \tk + b \tm \tn)un+t p\neq0$ and the number of solutions $(m,k)$ to
$$
ak + bmn=\frac{(a \tk + b \tm \tn)un+tp}{\tu \tn }
$$
with $m\leqslant M,k\leqslant K$ is at most $O\(1+K/N\)$ since $k$ falls in 
a prescribed  arithmetic progression modulo $n \sim N$, which is already fixed. 

Taking all possibilities into account, we find the contributions from $ak + bmn\neq 0$ to $\cI_2$ are at most
$$
KNTU\cdot M \cdot \(1+\frac{K}{N}\) p^{o(1)} \leqslant  KMTU(K+N)p^{o(1)}.
$$
Now Lemma~\ref{lm:I2-upperbound} follows by recalling the choice of $T$ in~\eqref{eq:T-choice}.
\end{proof}

To bound $\cI_3$, we apply Lemma~\ref{lm:DavenportErdos} with $\cA=\{bv:v\sim V\}$, getting
\begin{equation}\label{eq:S-uperbound}
\cI_3\ll V^{2r}p^{\frac{1}{2}}+V^rp
\end{equation}
as a consequence of Weil's bound for complete character sums over finite fields
(see~\cite[Theorem~11.23]{IK04}).

\subsection{Concluding the proof of Theorem~\ref{thm:TypeI}}
To balance the two terms in~\eqref{eq:S-uperbound}, we choose
$$
V=p^{\frac{1}{2r}} \mand U = \frac{1}{4} M p^{-\frac{1}{2r}}
$$
in view of~\eqref{eq:UV=M/4}, so that~\eqref{eq:I1-upperbound} and\eqref{eq:I2-upperbound} 
become
$$
\cI_1\ll KM^2Np^{-\frac{1}{2r}} \mand \cI_2
\leqslant  KM^2p^{-\frac{1}{2r}}(K+N)(1+\mathscr{L}_1MNp^{-1-\frac{1}{2r}})p^{o(1)}
$$
respectively, 
where $\mathscr{L}_1$ is defined by~\eqref{eq:L-notation}, while~\eqref{eq:S-uperbound} becomes 
$$
\cI_3\ll p^{\frac{3}{2}}.
$$

Now Theorem~\ref{thm:TypeI} follows by combining the inequality~\eqref{eq:T-I1I2S} and the 
above estimates for $\cI_1$, $\cI_2$ and $\cI_3$.

\section{Proof of Theorem~\ref{thm:TypeII}} 

\subsection{Cauchy--Schwarz inequality and amplification}
We note that the auxiliary parameters $T, U,V $ from \S~\ref{sec:Type-I proof} have different meaning henceforth.
This time it is convenient to assume that $\|\balpha\|_\infty\leqslant1$. 
By the Cauchy--Schwarz inequality, we have
\begin{equation}\label{eq:bilinearform-T}
|\fT(\balpha,\bbeta)|^2\leqslant \|\boldsymbol\beta\|_2^2T,
\end{equation}
where
$$
T=\sum_{k\leqslant K}\sum_{n\in\Z} W\(\frac{n}{N}\)
\left| \sum_{m\leqslant M}\alpha_m\chi(ak+bmn)\right|^2 
$$
with any fixed real-valued smooth function $W \in \cC_c^\infty([-2,2])$, which majorizes
the characteristic function of the unit interval $[0,1]$.

Squaring out, changing the order of summation and estimating the contribution 
from  the diagonal terms with $m_1 = m_2$,  we obtain
\begin{equation}\label{eq:T-T1}
T\ll T_1+KMN
\end{equation}
with
$$
T_1 =\sum_{k\leqslant K}\mathop{\sum\sum}_{\substack{m_1,m_2\leqslant M\\ m_1\neq m_2}}\sum_{n\in\Z}W\(\frac{n}{N}\)\alpha_{m_1}\overline{\alpha}_{m_2}\chi(ak+bm_1n)\overline{\chi}(ak+bm_2n).
$$
For any integers $u,v$ we may write
$$
T_1=\sum_{k\leqslant K} \mathop{\sum\sum}_{\substack{m_1,m_2\leqslant M\\ m_1\neq m_2}}\sum_{n\in\Z}W\(\frac{n+uv}{N}\)\alpha_{m_1} \overline{\alpha}_{m_2}\chi(ak+bm_1(n+uv))\overline{\chi}(ak+bm_2(n+uv)).
$$
By Fourier inversion and summing over $u\sim U,v\sim V$ for some positive parameters $U$ and $V$ with 
\begin{equation}\label{eq:UV=N/4}
U,V\geqslant1,\ \ \ \  UV=\frac{1}{4}N,
\end{equation}
it follows that
\begin{align*}
T_1&\ll\frac{1}{N}\sum_{k\leqslant K}\mathop{\sum\sum}_{\substack{m_1,m_2\leqslant M\\ m_1\neq m_2}}\sum_{|n|\leqslant 2N} \sum_{u\sim U}\int_\R |\widehat{W}(\xi)|\\
& \qquad  \qquad \times\biggl| \sum_{v\sim V}\ue\(\frac{uv\xi}{N}\)\chi(\overline{m_1u}(ak+bm_1n)+bv)\overline{\chi}(\overline{m_2u}(ak+bm_2n)+bv)\biggr|\ud\xi.
\end{align*}
In view of 
$$
\widehat{W}(\xi)\ll (1+|\xi|)^{-2}
$$
implied by the smoothness of $W(x)$ via partial integration, 
 and making the change of variable $\xi\rightarrow \xi/u$, after simple transformations, we obtain
\begin{align*}
T_1&\ll\frac{1}{NU}\int_\R \(1+\frac{|\xi|}{U}\)^{-2}\sum_{k\leqslant K}\mathop{\sum\sum}_{\substack{m_1,m_2\leqslant M\\ m_1\neq m_2}} \sum_{|n|\leqslant 2N}\sum_{u\sim U}\\
&\qquad  \qquad \times \biggl| \sum_{v\sim V}\ue\(\frac{v\xi}{N}\)\chi(\overline{m_1u}(ak+bm_1n)+bv)\overline{\chi}(\overline{m_2u}(ak+bm_2n)+bv)\biggr|\ud\xi.
\end{align*}
Again, as in the proof of Theorem~\ref{thm:TypeI},  this implies 
\begin{equation}\label{eq:T1-amplified}
\begin{split}
T_1&\ll\frac{1}{N}\sum_{k\leqslant K}\mathop{\sum\sum}_{\substack{m_1,m_2\leqslant M\\ m_1\neq m_2}}\sum_{|n|\leqslant 2N} \sum_{u\sim U}\\
&\qquad \quad \times\left| \sum_{v\sim V}\ue(v\xi)\chi(\overline{m_1u}(ak+bm_1n)+bv)\overline{\chi}(\overline{m_2u}(ak+bm_2n)+bv)\right|
\end{split}
\end{equation}
for some $\xi\in\R$. 

\subsection{Regrouping, counting and Weil's bound}
Put
$$
\varrho(x_1,x_2)=\mathop{\sum_{|n|\leqslant 2N}\sum_{k\leqslant K}\mathop{\sum\sum}_{m_1,m_2\leqslant M}\sum_{u\sim U}}_{\substack{ak+bm_1n\equiv um_1x_1 \not\equiv0\bmod p\\ ak+bm_1n\equiv um_2x_2\not\equiv0\bmod p\\ m_1\neq m_2}}1
$$
for $x_1,x_2 \bmod p$. We now have
$$
T_1\ll\frac{1}{N} \mathop{\sum\sum}_{x_1,x_2\bmod p}
 \varrho(x_1,x_2)\left| \sum_{v\sim V}\ue(v\xi)\chi(x_1+bv)\overline{\chi}(x_2+bv)\right|+(1+MN/p)KMp^{o(1)},
$$
where the last term comes from those $k,m_1,m_2,n$ with 
$$
p\mid(ak+bm_1n)(ak+bm_2n)
$$ 
in~\eqref{eq:T1-amplified}.

Applying the H\"older inequality with $r\ge 1$ similarly to~\eqref{eq:T-I1I2S}, we find
\begin{equation}\label{eq:T1-J1J2Simga}
T_1\ll \frac{1}{N}\cJ_1^{1-\frac{1}{r}}(\cJ_2\cdot \cJ_3)^{\frac{1}{2r}}+(1+MN/p)KMp^{o(1)}
\end{equation}
with
$$
\cJ_j=\mathop{\sum\sum}_{x_1,x_2\bmod p}
 \varrho(x_1,x_2)^j,\quad j=1,2,
$$
and
$$
\cJ_3=\mathop{\sum\sum}_{x_1,x_2\bmod p}\left| \sum_{v\sim V}\ue(v\xi)\chi(x_1+bv)\overline{\chi}(x_2+ bv)\right|^{2r}.
$$

Note that
\begin{equation}\label{eq:J1-upperbound}
\cJ_1\ll KM^2NU.
\end{equation}
We appeal to the following lemma in bounding $\cJ_2$.

\begin{lemma}\label{lm:J2-upperbound}
With the above notation, we have
\begin{equation}\label{eq:J2-upperbound}
\cJ_2
\leqslant  KNU(K+M^2)\(1+ \frac{\(|a|K + |b|MN\)MU}{p}\)^2p^{o(1)}.
\end{equation}  
\end{lemma}

\begin{proof}
We observe that $\cJ_2$ is bounded by the number of 10-tuples 
$$
(k,\tk ,m_1,\tm_1,m_2,\tn_2,n,\tn,u,\tu ) \in \N^{10}
$$ 
satisfying
\begin{align*}
k,\tk\leqslant K, \quad m_1,\tm_1,m_2,\tm_2&\leqslant M,\quad n,\tn \in[-2N,2N],\quad u,\tu \sim U,\\ 
(ak+bm_1n)\tm_1\tu &\equiv (a\tk+b\tm_1\tn)m_1u\not\equiv0\bmod p,\\
(ak+bm_2n)\tm_2\tu &\equiv (a\tk+b\tm_2\tn)m_2u\not\equiv0\bmod p,\\
m_1&\neq m_2,  \quad \tm_1\neq \tm_2.
\end{align*}
We now rewrite the above congruences as equations 
\begin{subequations}\begin{align}
(ak+bm_1n)\tm_1\tu &=(a\tk+b\tm_1\tn)m_1u+t_1p\ne 0, \label{eq:secondmoment-eq1}\\
(ak+bm_2n)\tm_2\tu &=(a\tk+b\tm_2\tn)m_2u+t_2p\ne 0,\label{eq:secondmoment-eq2}
\end{align}
\end{subequations}
where $0\leqslant |t_1|,|t_2|\leqslant T$ with 
\begin{equation}\label{eq:T-choice2}
T=1+(|a|K+2|b|MN)MU/p.
\end{equation}

It follows from~\eqref{eq:secondmoment-eq1} that $\tm_1\mid a\tk m_1u+t_1p\neq0$ and $\tu\mid  (\tm_1\tn +a\tk)m_1u+t_1p\neq0$, which produces at most $p^{o(1)}$ tuples of $(\tm_1,\tu)$ for given $\tn ,\tk,m_1,u,t_1$. 
After fixing $\tm_1,\tu,\tn ,\tk,m_1,u,t_1$, we obtain the equation
\begin{equation}\label{eq:equation-n,k1}
ak+bm_1n= \frac{(a\tk+b\tm_1\tn)m_1u+t_1p}{\tm_1\tu}
\end{equation}
in $k$ and $n$, thanks to~\eqref{eq:secondmoment-eq1}. 
We further fix $m_2$ and $t_2$. We see  from~\eqref{eq:secondmoment-eq2} that $\tm_2\mid a\tk m_2u+t_2p\neq0$, by which there are at most $p^{o(1)}$ values of $\tm_2$ for given $\tk,m_2,u,t_2$. As before, using~\eqref{eq:secondmoment-eq2}, 
we obtain the equation
\begin{equation}\label{eq:equation-n,k2}
ak+bm_2n= \frac{(a\tk+b\tm_2\tn)m_2u+t_2p}{\tm_2\tu}
\end{equation}
in $k$ and $n$. It then follows from~\eqref{eq:equation-n,k1} and~\eqref{eq:equation-n,k2} that
$k$ is uniquely defined modulo $m_1$ and modulo $m_2$, so that the number of such 
positive integers $k$ is at most
$$
 \frac{K}{\lcm[m_1,m_2]}+1=\frac{\gcd(m_1,m_2)}{m_1m_2}K+1.
$$
Now $n$ is uniquely determined after $k$ is fixed. Collecting all above arguments, we find
\begin{align*}
\cJ_2
&\le KNUT^2p^{o(1)}\mathop{\sum\sum}_{m_1,m_2\leqslant M}\(\frac{\gcd(m_1,m_2)}{m_1m_2}K+1\)\nonumber\\
&\le KNUT^2(K+M^2)p^{o(1)}.
\end{align*}
Lemma~\ref{lm:J2-upperbound} now follows immediately by recalling the choice of $T$ in~\eqref{eq:T-choice2}.
\end{proof}

Using Lemma~\ref{lm:DavenportErdos2} with the special choice $\cA=\{bv:v\sim V\}$ therein, we find
\begin{equation}\label{eq:Simga-upperbound}
\cJ_3\ll V^{2r}p+V^rp^2.
\end{equation}

\subsection{Concluding the proof of  Theorem~\ref{thm:TypeII}}
We now choose
$$
V = p^{\frac{1}{r}} \mand U =  \frac{1}{4}N  p^{-\frac{1}{r}}
$$
subject to the constraint in~\eqref{eq:UV=N/4},
so that the above bounds for $\cJ_1$ and $\cJ_2$ in~\eqref{eq:J1-upperbound} and~\eqref{eq:J2-upperbound}  become
$$
\cJ_1\ll K(MN)^2p^{-\frac{1}{r}}\mand
\cJ_2
\leqslant  KN^2p^{-\frac{1}{r}}(K+M^2)(1+\mathscr{L}_1MNp^{-1-\frac{1}{r}})^2p^{o(1)},
$$ 
respectively, while the bound for $\cJ_3$ in ~\eqref{eq:Simga-upperbound} 
becomes 
$$
\cJ_3\ll p^3.
$$
Substituting these estimates to~\eqref{eq:T1-J1J2Simga}, we arrive that
\begin{align*}
T_1
 \le\frac{1}{N}(KM^2N^2)^{1-\frac{1}{2r}}(1+KM^{-2})^{\frac{1}{2r}}&(1+\mathscr{L}_1MNp^{-1-\frac{1}{r}})^{\frac{1}{r}}\\
 &\qquad \times p^{\frac{1}{2r}+\frac{1}{2r^2}+o(1)}+(1+MN/p)KMp^{o(1)},
\end{align*}  
from which and then inequalities~\eqref{eq:bilinearform-T} and~\eqref{eq:T-T1}, Theorem~\ref{thm:TypeII} follows.

\section{Proof of Theorem~\ref{thm:quadrilinearform}} 

\subsection{Preliminary transformations}

Assume $\|\boldsymbol\alpha\|_\infty\leqslant1$ and $\gcd(a,b)=1$ without loss of generality.  By periodicity, we also assume that $1\leqslant |a|,|b|<p/2$.
By the Cauchy--Schwarz inequality, 
we have
\begin{equation}\label{eq:quadrilinear-Q}
|\fQ(\boldsymbol\alpha,\boldsymbol\beta)|^2\leqslant \|\boldsymbol\beta\|_2^2Q,
\end{equation}
where  
$$
Q=\mathop{\sum\sum}_{k,n\in\Z}	\varPhi\(\frac{k}{K}\)\varPhi\(\frac{n}{N}\)\left|\sum_{\ell\leqslant L}\sum_{m\leqslant M}\alpha_{\ell, m}\chi(ak \ell + bmn)\right|^2,
$$
for  any  smooth function   $\varPhi$, which dominates the characteristic function of $[-1,1]$
and is supported only inside the interval $[-2,2]$.
Squaring out and switching  the order of summation,   we get
\begin{equation}\label{eq:Q-Q1}
Q\leqslant |Q_1|+KLMN(LM/p+1)p^{o(1)}
\end{equation}
with
\begin{align*}
Q_1=\mathop{\sum\sum\sum\sum}_{\substack{\ell_1,\ell_2\leqslant L,~m_1,m_2\leqslant M\\ \ell_1m_2\not\equiv \ell_2m_1\bmod p}}\mathop{\sum\sum}_{k,n\in \Z} \varPhi\(\frac{k}{K}\)& \varPhi\(\frac{n}{N}\) \alpha_{\ell_1,m_1}\overline{\alpha}_{\ell_2,m_2}\\
&\quad  \times  \chi(ak\ell_1+bm_1n)\overline{\chi}(ak\ell_2+bm_2n),
\end{align*}
where the second term in~\eqref{eq:Q-Q1} comes from the contribution from $\ell_1m_2\equiv \ell_2m_1\bmod p$.
The trivial bound for $Q_1$ is
\begin{equation}\label{Q1trivial}
Q_1\ll KL^2M^2N,
\end{equation}
and, by~\eqref{eq:quadrilinear-Q}, we obtain a non-trivial bound of $\fQ(\boldsymbol\alpha,\boldsymbol\beta)$ as soon as we improve~\eqref{Q1trivial}. 
For all integers $u,v, w$, using that $\Z$ is invariant under shifts by integers,  we can write 
\begin{align*}
Q_1&=\mathop{\sum\sum\sum\sum}_{\substack{\ell_1,\ell_2\leqslant L,~m_1,m_2\leqslant M\\ \ell_1m_2\not\equiv \ell_2m_1\bmod p}}\mathop{\sum\sum}_{k,n\in \Z}\varPhi\(\frac{k+uv}{K}\)\varPhi\(\frac{n+wv}{N}\) \alpha_{\ell_1,m_1}\overline{\alpha}_{\ell_2,m_2}\\
&\qquad \qquad \quad\times\chi\left(a(k+uv)\ell_1+bm_1(n+wv)\right) \overline{\chi}\left(a(k+uv)\ell_2+bm_2(n+wv)\right).
\end{align*}
By Fourier inversion and summing over $u\sim U,v\sim V,w\sim W$ with 
\begin{equation}\label{eq:defUVW}
U,V,W\geqslant1,\ \  \ \ UV=\frac{1}{4}K \mand WV=\frac{1}{4}N,
\end{equation}
it follows that 
\begin{equation}
\begin{split}
\label{721}
Q_1&\ll\frac{1}{UVW}\sum_{|k|\leqslant2K}\sum_{|n|\leqslant2N}\mathop{\sum\sum\sum\sum}_{\substack{\ell_1,\ell_2\leqslant L,~ m_1,m_2\leqslant M\\ \ell_1m_2\not\equiv \ell_2m_1\bmod p}}\sum_{u\sim U}\sum_{w\sim W}\iint_{\R^2} |\widehat{\varPhi}(\xi)\widehat{\varPhi}(\eta)|\\
& \qquad  \qquad\quad \times\biggl|\sum_{v\sim V}\ue\(\frac{uv\xi}{K}+\frac{wv\eta}{N}\)\chi\left(\overline{(a\ell_1u+bm_1w)}(ak\ell_1+bm_1n)+v\right) \\
& \qquad \qquad \qquad \qquad  \quad    \times \overline{\chi}\left(\overline{(a\ell_2u+bm_2w)}(ak\ell_2+bm_2n)+v\right)\biggr|\ud\xi\ud\eta + \mathsf {Err}_1,
\end{split}
\end{equation}
where
\begin{itemize}
\item the variables of summation satisfy the extra condition 
\begin{equation}\label{eq:pdoesnotdivide}
p \nmid (a\ell_1u +bm_1w)(a\ell_2u+bm_2w)(ak\ell_1+bm_1n)(ak\ell_2+bm_2n),
\end{equation}
\item the term $\mathsf {Err}_1$  corresponds to the error induced by the terms 
$$p\mid (a\ell_1u +bm_1w)(a\ell_2u+bm_2w)(ak\ell_1+bm_1n)(ak\ell_2+bm_2n).$$
\end{itemize} 
As a typical possibility, the number of tuples $(\ell_1,u,m_1,w)$ satisfying $p\mid (a\ell_1u +bm_1w)$ is at most $LU(1+MW/p)p^{o(1)}$. Taking all the remaining possibilities into account, we find
\begin{equation}\label{boundforE}
\mathsf {Err}_1 \leqslant  \frac{KL^2MNUV}{UVW}\(1+\frac{MW}{p}\)p^{o(1)}
\leqslant  KL^2M\(V+\frac{MN}{p}\)p^{o(1)}.
\end{equation}

Combining~\eqref{boundforE} with
   the inequality $\widehat{\varPhi}(\xi)\ll (1+|\xi|)^{-2}$ and with  the change of variables $(\xi,\eta)\rightarrow (\xi/u,\eta/w)$ in~\eqref{721}, we obtain
\begin{align*}
Q_1&\ll\frac{1}{U^2VW^2}\iint_{\R^2} \(1+\frac{|\xi|}{U}\)^{-2}\(1+\frac{|\eta|}{W}\)^{-2}\sum_{|k|\leqslant 2K}\sum_{|n|\leqslant 2N}\mathop{\sum\sum\sum\sum}_{\substack{\ell_1,\ell_2\leqslant L,~ m_1,m_2\leqslant M\\ \ell_1m_2\not\equiv \ell_2m_1\bmod p}}\\
& \qquad \qquad \qquad   \times  \sum_{u\sim U}\sum_{w\sim W} \biggl|\sum_{v\sim V}\ue\(\frac{v\xi}{K}+\frac{v\eta}{N}\)\chi\(\overline{(a\ell_1u+bm_1w)}(ak\ell_1+bm_1n)+v\)\\
& \qquad \qquad \qquad  \qquad \qquad \quad   \qquad\times \overline{\chi}
\(\overline{(a\ell_2u+bm_2w)}(ak\ell_2+bm_2n)+v\)\biggr|\ud\xi\ud\eta 
\,   + \mathsf {Err}_1.
\end{align*}
This implies
\begin{align*}
Q_1\ll\frac{1}{UVW}\sum_{|k|\leqslant2 K}&\sum_{|n|\leqslant 2N}\mathop{\sum\sum\sum\sum}_{\substack{\ell_1,\ell_2\leqslant L,~ m_1,m_2\leqslant M\\ \ell_1m_2\not\equiv \ell_2m_1\bmod p}}\\
&\quad \times\sum_{u\sim U}\sum_{w\sim W}\Bigl|\sum_{v\sim V}\ue(v\xi)\chi\(\overline{(a\ell_1u+bm_1w)}(ak\ell_1+bm_1n)+v\)\\
& \qquad \qquad \quad \quad \times\overline{\chi}\(\overline
{(a\ell_2u+bm_2w)}(ak\ell_2+bm_2n)+v\)\Bigr|  + \mathsf {Err}_1
\end{align*}
for some $\xi\in\R$, where $\mathsf {Err}_1$ satisfies~\eqref{boundforE}. We now pull out the gcd of $ak$ and $bn$ for latter purpose.
To this end, we put 
\begin{align*}
d=\gcd(ak,bn),~~d_1 = \frac{d}{\gcd(a,d)},~~d_2=\frac{d}{\gcd(b,d)},~~a^* = \frac{a}{\gcd(a,d)},~~b^*=\frac{b}{\gcd(b,d)}.
\end{align*}
We observe that $p\nmid d$,  $\gcd(a,d)=\gcd(a,n)$, $\gcd(b,d)=\gcd(b,k)$, and $d\le D$ with
$$
D =\min\{|a|K,|b|N\}~~(<p^2).
$$
Note that $d\mid ak$ implies $d_1\mid k$, and similarly, $d\mid bn$ implies $d_2\mid n$. Hence, we have the relations $d_1\mid k$, $d_2\mid n$, $ad_1=a^*d$ and $bd_2=b^*d$.
Therefore, changing the variables 
\begin{align}\label{eq:changevariables:kn}
k \mapsto d_1k,\ \ \ n \mapsto d_2n,
\end{align}
 we obtain  the inequality 
\begin{equation}\label{eq:T1-T1(d)}
Q_1\ll\frac{1}{UVW}\sum_{\substack{d\leqslant D\\ p\nmid d}}Q_1(d) \,  + \mathsf {Err}_1 
\end{equation}
with
\begin{align*}
Q_1(d)
&=\mathop{\sum_{|k|\leqslant 2K/d_1}\sum_{|n|\leqslant 2N/d_2}}_{\gcd(a^*k,b^*n)=1}\mathop{\sum\sum\sum\sum}_{\substack{\ell_1,\ell_2\leqslant L,~ m_1,m_2\leqslant M\\ \ell_1m_2\not\equiv \ell_2m_1\bmod p}}\\&
\qquad \qquad\times\sum_{u\sim U}\sum_{w\sim W}\biggl| \sum_{v\sim V}\ue(v\xi)
\chi\(\overline{(a\ell_1u+bm_1w)}(ad_1k\ell_1+bd_2m_1n)+v\)\\
&\qquad \qquad\qquad \qquad\times\overline{\chi}\(\overline{(a\ell_2u+bm_2w)}(ad_1k\ell_2+bd_2m_2n)+v\)\biggr|.
\end{align*}
Writing $ad_1k\ell_1+bd_2m_1n= d (a^*k\ell_1+b^*m_1n)$, we now continue as
\begin{align*}
Q_1(d)
&=\mathop{\sum_{|k|\leqslant 2K/d_1}\sum_{|n|\leqslant 2N/d_2}}_{\gcd(a^*k,b^*n)=1}\mathop{\sum\sum\sum\sum}_{\substack{\ell_1,\ell_2\leqslant L,~ m_1,m_2\leqslant M\\ \ell_1m_2\not\equiv \ell_2m_1\bmod p}}\\
&
\qquad \times\sum_{u\sim U}\sum_{w\sim W}\biggl| \sum_{v\sim V}\ue(v\xi)
\chi\(d\overline{(a\ell_1u+bm_1w)}(a^*k\ell_1+b^*m_1n)+v\)\\
&\qquad \qquad\times\overline{\chi}\(d\overline{(a\ell_2u+bm_2w)}(a^*k\ell_2+b^*m_2n)+v\)\biggr|,
\end{align*}
which can be also interpreted as
\begin{align*}
Q_1(d)
&=\mathop{\sum_{|k|\leqslant 2K/d_1}\sum_{|n|\leqslant 2N/d_2}}_{\gcd(a^*k,b^*n)=1}\mathop{\sum\sum\sum\sum}_{\substack{\ell_1,\ell_2\leqslant L,~ m_1,m_2\leqslant M\\ \ell_1m_2\not\equiv \ell_2m_1\bmod p}}\\
&\qquad \times\sum_{\substack{u\sim \gcd(a,d) U\\ \gcd(a,d)\mid u}}\sum_{\substack{w\sim \gcd(b,d) W\\ \gcd(b,d)\mid w}}\biggl| \sum_{v\sim V}\ue(v\xi)
\chi\(d\overline{(a^*\ell_1u+b^*m_1w)}(a^*k\ell_1+b^*m_1n)+v\)\\
&\qquad \qquad\times\overline{\chi}\(d\overline{(a^*\ell_2u+b^*m_2w)}(a^*k\ell_2+b^*m_2n)+v\)\biggr|.
\end{align*}
Here we have used the formulas $a^*=a/\gcd(a,d)$, $b^*=b/\gcd(b,d)$.
Of course the variables of summation continue to satisfy~\eqref{eq:pdoesnotdivide} with necessary changes of variables as in \eqref{eq:changevariables:kn}.
We further impose a restriction that $kw\gcd(b,d)\neq nu\gcd(a,d)$,  which introduces  an additional error $\mathsf {Err}_2 (d)$ with
$$\sum_{d\leqslant D}\mathsf {Err}_2 (d) \leqslant KL^2M^2VW p^{o(1)}.$$
Next for $x_1,x_2 \in \F_p^*$, and for 
$$1\leqslant A\ll \mathscr{L}_2/d,\quad  1\leqslant B\ll |a|LU+|b|MW\ll \mathscr{L}_2/V, \quad  1\leqslant C\leqslant K/d_1,$$ 
we put 
\begin{align}\label{eq:rho(x1,x2)}
\varrho(x_1,x_2)=\mathop{\sum_{|k|\sim C}\sum_{|n|\leqslant 2N/d_2}\mathop{\sum\sum}_{\ell_1,\ell_2\leqslant L}\mathop{\sum\sum}_{m_1,m_2\leqslant M}\sum_{u\sim U}\sum_{w\sim W}}_{\substack{a^*k\ell_1+b^*m_1n\equiv (a \ell_1u+b m_1w)x_1\bmod p\\ a^*k\ell_2+b^*m_2n\equiv (a \ell_2u+b m_2w)x_2\bmod p\\ \ell_1m_2\not\equiv \ell_2m_1\bmod p,\quad kw\gcd(b,d)\neq nu\gcd(a,d),\quad\gcd(a^*k,b^*n)=1\\
|a^*k\ell_1+b^*m_1n|\sim A,~|a \ell_2u+b m_2w|\sim B}}1, 
\end{align}
where the condition~\eqref{eq:pdoesnotdivide} continues to apply to the variables of summation with necessary changes of variables as in \eqref{eq:changevariables:kn}.
 We now have
\begin{align*}
Q_1(d)\leqslant p^{o(1)}\sup_{\substack{1\leqslant A\ll \mathscr{L}_2/d\\
1\leqslant B\ll \mathscr{L}_2/V\\1\leqslant C\leqslant K/d_1}}\, 
\sum_{x_1,x_2 \in \F_p^*}  
\varrho(x_1,x_2)\left|\sum_{v\sim V}\ue(v\xi)\chi(dx_1+v) \overline{\chi}(dx_2+v)\right|+
\mathsf {Err}_2 (d).
\end{align*} 

Applying the H\"older inequality with a positive integer $r\geqslant2$, we find 
\begin{equation}\label{eq:T0(d)-Holder}
Q_1(d)\leqslant  Q_2(d) \,  p^{o(1)}  + \mathsf {Err}_2 (d),
\end{equation}
where 
\begin{equation}\label{eq:Q2(d)}
Q_2(d) = \sup_{\substack{1\leqslant A\ll \mathscr{L}_2/d\\
1\leqslant B\ll \mathscr{L}_2/V\\1\leqslant C\leqslant K/d_1}}\Sigma_1^{1-\frac{1}{r}}\(\Sigma_2\cdot \Sigma_3\)^{\frac{1}{2r}}
\end{equation}
with
$$
\Sigma_j=\Sigma_j(A,B,C) =
\mathop{\sum\sum}_{x_1,x_2 \in \F_p^*}  \varrho(x_1,x_2;A,B,C)^j,\qquad  j=1,2,
$$
and
$$
\Sigma_3= 
\mathop{\sum\sum}_{x_1,x_2 \in \F_p} 
\left|\sum_{v\sim V}\ue(v\xi)\chi(x_1+v)\overline{\chi}(x_2+v)\right|^{2r}.
$$

Trivially we have 
\begin{equation}\label{eq:Sigma1-sup}
\Sigma_1\ll CL^2M^2NUWd_2^{-1}.
\end{equation}
Furthermore, using Lemma~\ref{lm:DavenportErdos2}, we have the inequality
\begin{equation}\label{eq:Sigma3-sup}
\Sigma_3\ll V^{2r}p+V^rp^2.
\end{equation}
Next we estimate $\Sigma_2$.

\subsection{Bounding $\Sigma_2$}

This following bound is the core of our method.

\begin{lemma}\label{lm:Sigma2-upperbound}
With the above notation, we have
$$
\Sigma_2
\le (M/C+1)^2\(1+\frac{\mathscr{L}_2^2}{dpV}\)^2(ABLM)^{1+o(1)}(CW\gcd(b,d)+NU\gcd(a,d)/d_2)
$$
for all  $1\leqslant A\ll \mathscr{L}_2/d$, $1\leqslant B\ll \mathscr{L}_2/V$ and $1\leqslant C \leqslant K/d_1$.
\end{lemma}

\begin{proof}
Let $U_1=\gcd(a,d) U$ and $W_1=\gcd(b,d) W$. We operate the changes of variables 
\begin{align}\label{eq:changevariables:uw}
u\mapsto \frac{u}{\gcd(a, d)},\ \ w\mapsto \frac{w}{\gcd(b, d)}
\end{align}
 in the definition \eqref{eq:rho(x1,x2)} of $\rho(x_1,x_2)$.
Note that $\Sigma_2$ is bounded by the number of tuples with length 16
$$
(k,\widetilde k,\ell_1,\widetilde \ell_1,\ell_2,\widetilde \ell_2,m_1,\widetilde m_1,
m_2,\widetilde m_2, n,\widetilde n,u,\widetilde u,w,\widetilde w)
$$ 
satisfying
\begin{align*}
|k|,|\widetilde k|\sim  C,&\quad  \ell_1,\widetilde \ell_1,\ell_2,\widetilde \ell_2\leqslant L, \quad m_1,\widetilde m_1,m_2,\widetilde m_2 \leqslant M,\quad  |n|,| \widetilde n|\leqslant 2N/d_2,\nonumber\\
  u,\widetilde u \sim U_1,&\quad w,\widetilde w\sim W_1, \quad  kw\neq nu,\quad  \widetilde k\widetilde w\neq  \widetilde n \widetilde u,\quad\gcd(a^*k,b^*n)=\gcd(a^*\widetilde k, b^*\widetilde n)=1,\nonumber\\
  |a^*k\ell_1+b^*m_1n|&\sim A,\quad |a^*\widetilde k\widetilde \ell_1+  b^*\widetilde m_1  \widetilde n|\sim A, \quad |a^*\ell_2u+b^*m_2w|\sim B,\quad |a^*\widetilde \ell_2\widetilde u+  b^*\widetilde m_2 \widetilde w|\sim B,\nonumber\\
&  \qquad \ell_1m_2\not\equiv \ell_2m_1\bmod p, \qquad  \widetilde \ell_1    \widetilde m_2 \not\equiv \widetilde \ell_2   \widetilde m_1\bmod p,
\end{align*}
with the additional non-divisbility conditions
\begin{equation}
 p \nmid (a^*\ell_1u +b^*m_1w)(a^*\ell_2u+b^*m_2w)(a^*k\ell_1+b^*m_1n)(a^*k\ell_2+b^*m_2n),\label{eq:pdoesnotdivide-afterchangevariables1}
 \end{equation}
and 
\begin{equation}
p \nmid (a^*\widetilde \ell_1\widetilde u +b^*\widetilde m_1\widetilde w)(a^*\widetilde \ell_2\widetilde u+b^*\widetilde m_2\widetilde w)(a^*\widetilde k\widetilde \ell_1+b^*\widetilde m_1\widetilde n)(a^*\widetilde k\widetilde \ell_2+b^*\widetilde m_2\widetilde n),\label{eq:pdoesnotdivide-afterchangevariables2}
\end{equation}
and also such that 
\begin{equation}
(a^*k\ell_1+b^*m_1n)(a^*\widetilde \ell_1\widetilde u+b^*\widetilde m_1\widetilde w)=(a^*\widetilde k\widetilde \ell_1+b^*\widetilde m_1\widetilde n)(a^*\ell_1u+b^*m_1w)+z_1p,\label{eq:diophantine1}
\end{equation}
and 
\begin{equation}
(a^*k\ell_2+b^*m_2n)(a^*\widetilde \ell_2\widetilde u+b^*\widetilde m_2\widetilde w)=(a^*\widetilde k\widetilde \ell_2+b^*\widetilde m_2\widetilde n)(a^*\ell_2u+b^*m_2w)+z_2p,\label{eq:diophantine2}
\end{equation}
with some  $0\leqslant |z_1|,|z_2|\leqslant  Z$
where 
$$Z=1+4(|a|KL+|b|MN)(|a|LU+|b|MW)/(dp),$$
so by hypothesis we have the inequality
\begin{equation}\label{ineqfor Z}
Z\ll 1 +  \frac{\mathscr{L}_2^2}{dpV}.
\end{equation}
Note that the conditions \eqref{eq:pdoesnotdivide-afterchangevariables1} and \eqref{eq:pdoesnotdivide-afterchangevariables2} are resulted by \eqref{eq:pdoesnotdivide} with changes of variables as in \eqref{eq:changevariables:kn} and \eqref{eq:changevariables:uw}.

We now fix $k,\ell_1,\ell_2,m_1,m_2,n,u,w$, and put
$$
a_j=a^*k\ell_j+b^*m_jn,\quad b_j=a^*\ell_ju+b^*m_jw
$$
for $j=1,2$. Note that $$|a_1|\sim A,\quad |a_2|\leqslant2\mathscr{L}_2/d, \quad |b_1|\leqslant\mathscr{L}_2/V, \quad  |b_2|\sim B,
$$ with $a_1a_2b_1b_2\neq0$. Given $z_1,z_2$ with $0\leqslant |z_1|,|z_2|\leqslant  Z$ as above, we now look at the equations
\begin{equation}\label{eq:Diophantine-a1a2b1b2}
a_1x_1=b_1y_1+z_1p,\quad a_2x_2=b_2y_2+z_2p
\end{equation}
in $x_1,x_2,y_1,y_2\in \Z$ with $|x_1|\leqslant \mathscr{L}_2/V$, $|x_2|\sim B$, $|y_1|\sim A$ and $|y_2|\leqslant 2\mathscr{L}_2/d$. It is clear that the number of such tuples $(x_1,x_2,y_1,y_2)$ satisfying~\eqref{eq:Diophantine-a1a2b1b2} is bounded, up to an absolute constant, by 
\begin{equation}
\label{eq:gcd gcd}
\(\frac{A\gcd(a_1,b_1)}{a_1}+1\)\(\frac{B\gcd(a_2,b_2)}{b_2}+1\)
\ll\gcd(a_1,b_1)\gcd(a_2,b_2).
\end{equation}

We claim that, for given $k,\ell_1,\ell_2,m_1,m_2,n,u,w,z_1,z_2$ as above, the number $\cN$ of tuples $(\widetilde k,\widetilde \ell_1,\widetilde \ell_2,\widetilde m_1,\widetilde m_2,\widetilde n,\widetilde u,\widetilde w)$ satisfying the above-mentioned conditions satisfies
\begin{align}\label{eq:numberof8variables}
\cN\ll (M/C+1)^2.
\end{align}
Hence
\begin{align*}
\Sigma_2
&\ll (M/C+1)^2 Z^2\mathop{\sum_{|k|\sim C}\mathop{\sum\sum}_{\ell_1,\ell_2\leqslant L}\mathop{\sum\sum}_{m_1,m_2\leqslant M}\sum_{|n|\leqslant 2N/d_2}\sum_{u\sim U_1}\sum_{w\sim W_1}}_{\substack{a^*k\ell_1+b^*m_1n=a_1,~a^*\ell_1u+b^*m_1w=b_1\\ a^*k\ell_2+b^*m_2n=a_2,~a^*\ell_2u+b^*m_2w=b_2\\|a_1|\sim A,~|b_2|\sim B,~kw\neq nu}}\gcd(a_1,b_1)\gcd(a_2,b_2)
\end{align*} 
in view of \eqref{eq:gcd gcd}. Then the desired bound for $\Sigma_2$ follows from~\eqref{ineqfor Z} and Lemma~\ref{lm:gcdsum}. 

It suffices to prove \eqref{eq:numberof8variables}, and keep in mind that $k,\ell_1,\ell_2,m_1,m_2,n,u,w,z_1,z_2$ are all fixed.
We now  fix one of such  tuples $(x_1,x_2,y_1,y_2)$, satisfying~\eqref{eq:Diophantine-a1a2b1b2}.
Then~\eqref{eq:diophantine1} and~\eqref{eq:diophantine2} lead us to consider the system of four equations
\begin{equation}\label{eq:Diophantine-x1x2y1y2}
\begin{cases}
a^*\widetilde \ell_1\widetilde u+b^*\widetilde m_1\widetilde w=x_1,\\
a^*\widetilde k\widetilde \ell_1+b^*\widetilde m_1\widetilde n=y_1,\\
a^*\widetilde \ell_2\widetilde u+b^*\widetilde m_2\widetilde w=x_2,\\
a^*\widetilde k\widetilde \ell_2+b^*\widetilde m_2\widetilde n=y_2, 
\end{cases}
\end{equation}
in $\widetilde k,\widetilde \ell_1,\widetilde \ell_2,\widetilde m_1,\widetilde m_2,\widetilde n,\widetilde u,\widetilde w\in \Z$. 
Recalling the restriction
$\widetilde \ell_1   \widetilde m_2 \not\equiv \widetilde \ell_2   \widetilde m_1\bmod p$,
it suffices to consider the solutions satisfying $\widetilde \ell_1\widetilde m_2\neq \widetilde \ell_2\widetilde m_1$.

We now fix integers $\fl_1,\fl_2,\fm_1$ and $\fm_2$ satisfying
\begin{equation}\label{eq:m1m2l1l2fixed}
\fl_1,\fl_2\leqslant L,\quad \fm_1,\fm_2\leqslant M,\quad 
\fl_1\fm_2\neq \fl_2\fm_1.
\end{equation}
such that the two equations
\begin{equation}\label{eq:m1m2l1l2fixed-equation(k,n)}
a^*\widetilde k\fl_1+b^*\fm_1\widetilde n=y_1,\quad a^*\widetilde k\fl_2+b^*\fm_2\widetilde n=y_2
\end{equation}
are solvable in $\widetilde k,\widetilde n$ with $\gcd(\widetilde k,\widetilde n)=1$, given the above $\fl_1,\fl_2,\fm_1,\fm_2$. Note that the system~\eqref{eq:m1m2l1l2fixed-equation(k,n)} has at most one solution $(\widetilde k,\widetilde n)$ because its determinant does not vanish. Suppose that a solution to~\eqref{eq:m1m2l1l2fixed-equation(k,n)} does exist. Then by~\eqref{eq:m1m2l1l2fixed} all solutions to the second equation in~\eqref{eq:Diophantine-x1x2y1y2} are of the shape
$$
\widetilde \ell_1=\fl_1-s_1b^*\widetilde n, \quad  \widetilde m_1=\fm_1+s_1a^*\widetilde k,\quad 0\leqslant |s_1|
\ll M/|\widetilde k|
\ll M/C,
$$
and all solutions to the fourth equation in~\eqref{eq:Diophantine-x1x2y1y2} are of the shape
$$
\widetilde \ell_2= \fl_2-s_2b^*\widetilde n,\quad  \widetilde m_2=\fm_2+s_2a^*\widetilde k,\quad   0\leqslant |s_2|
\ll M/|\widetilde k|
\ll M/C.
$$
After determining $\widetilde \ell_1,\widetilde \ell_2, \widetilde m_1,\widetilde m_2$ with $\widetilde \ell_1\widetilde m_2\neq \widetilde \ell_2\widetilde m_1$, we may find at most one tuple $(\widetilde u,\widetilde w)$ satisfying the first and third equations in~\eqref{eq:Diophantine-x1x2y1y2} simultaneously.

We have proved \eqref{eq:numberof8variables} so far, and thus completed the proof of Lemma~\ref{lm:Sigma2-upperbound}.
\end{proof}

Subsequently,  by Lemma~\ref{lm:Sigma2-upperbound}, for all $1\leqslant A\ll \mathscr{L}_2/d$, $1\leqslant B\ll |a|LU+|b|MW\ll \mathscr{L}_2/V$ and $1\leqslant C \leqslant K/d_1$, we have the bound 
\begin{equation}\label{eq:Sigma2-sup}
\begin{split}
\Sigma_2
\le (M/C+1)^2\(1+\frac{\mathscr{L}_2^2}{dpV}\)^2&\frac{(\mathscr{L}_2^2LM)^{1+o(1)}}{dV}\\
& \left(CW\gcd(b,d)+NU\gcd(a,d)/d_2\right).
\end{split}
\end{equation}

\subsection{Concluding the proof of  Theorem~\ref{thm:quadrilinearform}}
We substitute  the bounds~\eqref{eq:Sigma1-sup},  \eqref{eq:Sigma3-sup} and~\eqref{eq:Sigma2-sup} 
into~\eqref{eq:Q2(d)} and 
note that 
$$
d_1d_2 = \frac{d^2}{\gcd(a,d)\gcd(b,d)}  =  \frac{d^2}{\gcd(ab,d)} 
$$
since we have assumed $\gcd(a,b)=1$. 
Hence, we derive  
\begin{align*}
Q_2(d)
&\le p^{o(1)}\sup_{\substack{C\leqslant K/d_1}}\(CL^2M^2NUWd_2^{-1}\)^{1-\frac{1}{r}}(M/C+1)^{\frac{1}{r}}\(1+\frac{\mathscr{L}_2^2}{dpV}\)^{\frac{1}{r}}\(\frac{\mathscr{L}_2^2LM}{dV}\)^{\frac{1}{2r}}\\
& \qquad \qquad   \qquad   \qquad   \qquad  \times \(CW\gcd(b,d)+NU\gcd(a,d)/d_2\)^{\frac{1}{2r}}(V^{2r}p+V^rp^2)^{\frac{1}{2r}}.
\end{align*}
It is easy to see that in the last expression, after expanding, $C$ appears only in positive powers, and therefore the supremum
is attained at $C = K/d_1$. Hence,
\begin{align*}
Q_2(d)
&\le p^{o(1)}\(KL^2M^2NUW/d^2\)^{1-\frac{1}{r}}(d_1M/K+1)^{\frac{1}{r}}\(1+\frac{\mathscr{L}_2^2}{dpV}\)^{\frac{1}{r}}
\(\frac{\mathscr{L}_2^2LM}{dV}\)^{\frac{1}{2r}}\\
& \qquad  \times \gcd(ab,d)^{1-\frac{1}{r}} \(KW\gcd(b,d)/d_1+NU\gcd(a,d)/d_2\)^{\frac{1}{2r}}(V^{2r}p+V^rp^2)^{\frac{1}{2r}}\\
&\le p^{o(1)}\(KL^2M^2NUW/d^2\)^{1-\frac{1}{r}}(d_1M/K+1)^{\frac{1}{r}}\(1+\frac{\mathscr{L}_2^2}{dpV}\)^{\frac{1}{r}}\(\frac{\mathscr{L}_2^2KLMN}{d^2V^2}\)^{\frac{1}{2r}}\\
& \qquad \qquad \qquad \qquad \qquad \qquad  \qquad   \qquad   \qquad   \quad   \times \gcd(ab,d)^{1-\frac{1}{2r}}(V^{2r}p+V^rp^2)^{\frac{1}{2r}}.
\end{align*}
We also note that  
$$
d^{-2+\frac{2}{r}}\cdot d_1^{\frac{1}{r}}\cdot d^{-\frac{1}{r}}\cdot \gcd(ab,d)^{1-\frac{1}{2r}}\leqslant d^{-2+\frac{2}{r}}\gcd(ab,d)^{1-\frac{1}{2r}}\leqslant d^{-1}\gcd(ab,d)^{1-\frac{1}{2r}}
$$ 
fo $r \ge 2$. 
Recalling~\eqref{boundforE}, \eqref{eq:T1-T1(d)} and~\eqref{eq:T0(d)-Holder} and then summing over $d\leqslant D$, we obtain
\begin{align*}
Q_1
&\ll\frac{p^{o(1)}}{UVW}\sum_{d\leqslant D}\(Q_2(d)+ \mathsf {Err}_2 (d)\)+ \mathsf {Err}_1 \\
&\le p^{o(1)}\(KL^2M^2N\)^{1-\frac{1}{r}}(UVW)^{-\frac{1}{r}}(M/K+1)^{\frac{1}{r}}\(1+\frac{\mathscr{L}_2^2}{pV}\)^{\frac{1}{r}}\\
&  \qquad \qquad \qquad \qquad \qquad \qquad  \times \(\mathscr{L}_2^2KLMNp\)^{\frac{1}{2r}}(1+V^{-\frac{1}{2}}p^{\frac{1}{2r}})\\
&\qquad \qquad \qquad \qquad \qquad \qquad \qquad  +KL^2M^2U^{-1}p^{o(1)}+KL^2M\(V+\frac{MN}{p}\)p^{o(1)}.
\end{align*}
Taking 
$$
V=p^{\frac{1}{r}},  \qquad  U=\frac{1}{4}Kp^{-\frac{1}{r}}, \qquad W=\frac{1}{4}Np^{-\frac{1}{r}},
$$
so that the assumption~\eqref{eq:defUVW} is satisfied, we derive that
\begin{align*}
Q_1
&\le \(KL^2M^2N\)^{1-\frac{1}{r}}(KN)^{-\frac{1}{r}}(M/K+1)^{\frac{1}{r}}\(1+\frac{\mathscr{L}_2^2}{p^{1+\frac{1}{r}}}\)^{\frac{1}{r}}\(\mathscr{L}_2^2KLMN\)^{\frac{1}{2r}}p^{\frac{1}{2r}+\frac{1}{r^2}+o(1)}\\
&\qquad \qquad   \qquad   \qquad   \qquad   \qquad   \qquad  +(LM)^2p^{\frac{1}{r}+o(1)}+KL^2M\(p^{\frac{1}{r}}+\frac{MN}{p}\)p^{o(1)}\\
&\le \(KLMN\)^{2-\frac{3}{2r}}(KN)^{-1}(M/K+1)^{\frac{1}{r}}\(1+\frac{\mathscr{L}_2^2}{p^{1+\frac{1}{r}}}\)^{\frac{1}{r}}\mathscr{L}_2^{\frac{1}{r}}p^{\frac{1}{2r}+\frac{1}{r^2}+o(1)}\\
&\qquad \qquad  \qquad   \qquad  \qquad   \qquad   \qquad   +(LM)^2p^{\frac{1}{r}+o(1)}+KL^2M\(p^{\frac{1}{r}}+\frac{MN}{p}\)p^{o(1)}.
\end{align*}
It remains to recall~\eqref{eq:quadrilinear-Q} and~\eqref{eq:Q-Q1}, to  conclude the proof of Theorem~\ref{thm:quadrilinearform}.

\section{Proof of Theorem~\ref{thm:IS-analogue}}\label{sec:proofofIS-analogue}
Let $p>\max\{K,M,N\}$.
Consider the trilinear sum of Legendre symbols
$$
\fS=\sum_{k\leqslant K}\sum_{m\in\cM}\sum_{n\in\cN}\(\frac{k+mn}{p}\)
$$
with 
\begin{equation}\label{eq:K-choice}
K=\floor{p^{1+\frac {1}{r}+\eta}(MN)^{-2}+p^\eta}.
\end{equation}
Suppose that the values of $k+mn$, for all given $k,m,n$
in the above ranges, are always quadratic non-residues or zero modulo $p$. Then
we see from~\eqref{lowerdensity} that $\fS$ satisfies the trivial bound
$$
|\fS|\geqslant c_0^2 KMN+O((MN/p+1)Kp^{o(1)}),
$$
where the error term accounts for the cases with $k+mn \equiv 0 \pmod p$.

Therefore, Theorem~\ref{thm:IS-analogue} is proved as soon as we have the bound
\begin{equation}\label{eq:S=o(KMN)}
\fS= o(KMN)
\end{equation}
under the hypotheses~\eqref{condforMandN} and $p$ tending to infinity.

We first consider the case $(MN)^2> p^{1+\frac 1r}$, for which we see from~\eqref{eq:K-choice} that $K\asymp p^\eta$. Now~\eqref{eq:S=o(KMN)} is an immediate consequence of Karatsuba presented already in~\eqref{eq:Karatsuba-implicit} and~\eqref{eq:PV-threshold}. We henceforth assume that $(MN)^2\leqslant p^{1+\frac 1r}$.
Applying Theorem~\ref{thm:TypeII} with the choice
$$ 
a=b=1,\qquad \alpha_m= {\bf 1}_{\cM}(m), \qquad \beta_{k,n} ={\bf 1}_{[1,K]}(k) \cdot {\bf 1}_{\cN}(n),
$$
now~\eqref{eq:S=o(KMN)} is proved as soon as one has the following three inequalities
$$
M\geq K^{\frac 12},\qquad  
K(MN)^2 \geq p^{1+\frac 1r +\frac \eta 2},\qquad 
M\geq p^{\frac{\eta}4}.
$$

The choice~\eqref{eq:K-choice}
guarantees the above three conditions provided that
$$
M\geq p^{\frac{\eta}4} \mand  M^4N^2>p^{1+\frac{1}{r}+\eta}.
$$
The restriction $M\geq p^{\frac{\eta}4}$ can also be dropped, since otherwise we should have $N^2>p^{1+\frac{1}{r}}$,
and we may instead appeal to Karatsuba by noting that $K\geqslant p^\eta$ thanks to the choice~\eqref{eq:K-choice}.
This completes the proof of Theorem~\ref{thm:IS-analogue}.

\appendix
\section{Karatsuba's bound for double character sums}
\label{appendix:Karatsuba} 

Recall that
$$
\fB(\balpha,\bbeta) = \sum_{m\in \cM} \sum_{n \in \cN} \alpha_m\beta_n\chi(m + n)
$$
as in~\eqref{eq:bilinearform-Vinogradov}. We now give a proof of~\eqref{eq:Karatsuba-implicit} in the range of~\eqref{eq:PV-threshold}. More precisely, we have the following explicit estimate for $\fB(\balpha,\bbeta)$.

\begin{theorem}\label{thm:Karatsuba}
Let $\chi$ be a non-trivial character of $\F_p^\times$ and $\cM,\cN\subseteq\F_p$ two arbitrary subsets with cardinalities $M,N$, respectively. For each positive integer $r$, we have
$$
\fB(\balpha,\bbeta)
\ll \|\balpha\|_1^{1-\frac{1}{r}}\|\balpha\|_2^{\frac{1}{r}}\|\bbeta\|_\infty(Np^{\frac{1}{4r}}+N^{\frac{1}{2}}p^{\frac{1}{2r}}),
$$
where the implicit constant depends only on $r$.
\end{theorem}

\begin{remark}
Theorem~\ref{thm:Karatsuba} implies that
$$
\fB(\balpha,\bbeta)
\ll \|\balpha\|_\infty\|\bbeta\|_\infty MN(M^{-\frac{1}{2r}}p^{\frac{1}{4r}}+M^{-\frac{1}{2r}}N^{-\frac{1}{2}}p^{\frac{1}{2r}}).
$$
This readily implies~\eqref{eq:Karatsuba-implicit} for
$$
M>p^{\frac{1}{2}+\eta} \mand N>p^\eta
$$
by taking $r>1/\eta$.
\end{remark}

The proof is quite short.
From  the H{\"o}lder inequality, it follows that
$$
|\fB(\balpha,\bbeta)|   \le \sum_{m\in \cM} |\alpha_m|  \left|\sum_{n \in \cN} \beta_{n}\chi(m + n)\right|\le \|\balpha\|_1^{1-\frac{1}{r}}\|\balpha\|_2^{\frac{1}{r}}W^{\frac{1}{2r}},
$$
where 
$$
W=\sum_{m\in\cM}\left|\sum_{n \in \cN} \beta_{n}\chi(m + n)\right|^{2r}.
$$
The ingredient here is that one enlarges the sum over $m$ to $\F_p$, and ignores the structure of $\cM$, so that
$$
W\leqslant\sum_{m\in\F_p}\left|\sum_{n \in \cN} \beta_{n}\chi(m + n)\right|^{2r}.
$$
Now Theorem~\ref{thm:Karatsuba} follows immediately from Lemma~\ref{lm:DavenportErdos}.

\section*{Acknowledgements}

During the preparation of this work I.S. was supported in part by the  
Australian Research Council  Grants DP230100530 and DP230100534, and 
P.X. was supported in part by the National Natural Science Foundation of China (No.~12025106).

\bibliographystyle{plainnat}

\end{document}